\documentclass[aap,MSNbibl,dvips]{arximspdf}
\usepackage{mathrsfs}
\usepackage{graphicx}

\doi{10.1214/12-AAP912} 
\volume{23}
\issue{6}
\pubyear{2013}
\firstpage{2560}
\lastpage{2603}

\makeatletter

\newcommand{\rrVert}{\Vert}
\newcommand{\rrvert}{\vert}
\newcommand{\llVert}{\Vert}
\newcommand{\llvert}{\vert}
\def\cal{\mathcal}
\newtheorem{thmm}{Theorem}
\newproclaim{defn}[thmm]{Definition}
\newtheorem{lem}[thmm]{Lemma}
\newtheorem{prop}[thmm]{Proposition}
\newtheorem{cor}[thmm]{Corollary}
\newproclaim{fact}{Fact}
\newproclaim{assume}{Assumption}

\newcommand{\N}{\mathbb{N}}

\newcommand{\R}{\mathbb{R}}

\newcommand{\Co} {\mathcal{C}[0,1]}
\newcommand{\Do} {\mathcal{D}[0,1]}

\newcommand{\Drn}{\mathcal{D}_{r_n}[0,1]}

\newcommand{\p}[1]{\mathbf{P}(#1)}

\newcommand{\Ec}[1]{\mathbf{E} [{#1}]}
\newcommand{\eqdist}{\stackrel{d}{=}}
\newcommand{\E}[1]{\mathbf{E} [{#1} ]}
\newcommand{\V}[1]{\operatorname{Var} ({#1} )}
\newcommand{\Vc}[1]{\operatorname{Var}({#1})}
\newcommand{\Prob}[1]{\mathbf{P} ({#1} )}
\newcommand{\I}[1]{\mathbf{1}_{ \{ {#1} \} }}

\newcommand{\vol}{\operatorname{Leb}}
\newcommand{\gest}{\ge_{\mathrm{st}}}

\newcommand{\eqref}[1]{(\ref{#1})}
\makeatother

\begin{document}
\begin{frontmatter}

\title{A limit process for partial match queries in random quadtrees and $2$-d trees}
\runtitle{Partial match queries in random quadtrees}

\begin{aug}
\author[A]{\fnms{Nicolas} \snm{Broutin}\corref{}\ead[label=e1]{nicolas.broutin@inria.fr}},
\author[B]{\fnms{Ralph} \snm{Neininger}\ead[label=e2]{neiningr@math.uni-frankfurt.de}}
\and
\author[B]{\fnms{Henning} \snm{Sulzbach}\ead[label=e3]{sulzbach@math.uni-frankfurt.de}}
\runauthor{N. Broutin, R. Neiningern and H. Sulzbach}
\affiliation{Inria Rocquencourt, J.~W. Goethe-Universit\"at Frankfurt
and J.~W.~Goethe-Universit\"at Frankfurt}
\address[A]{N. Broutin\\
Inria Paris--Rocquencourt\\
Domaine de Voluceau\\
78153 Le Chesnay\\
France\\
\printead{e1}}
\address[B]{R. Neininger\\
H. Sulzbach\\
Institut f\"ur Mathematik\\
J.~W. Goethe-Universit\"at\\
60054 Frankfurt a.M.\\
Germany\\
\printead{e2}\\
\phantom{E-mail:\ }\printead*{e3}}
\end{aug}

\received{\smonth{2} \syear{2012}}
\revised{\smonth{11} \syear{2012}}

%
\begin{abstract}We consider the problem of recovering items matching a
partially specified pattern in multidimensional trees (quadtrees and
$k$-d trees). We assume the traditional model where the data consist of
independent and uniform points in the unit square. For this model, in a
structure on $n$ points, it is known that the number of nodes $C_n(\xi
)$ to visit in order to report the items matching a random query $\xi$,
independent and uniformly distributed on $[0,1]$, satisfies $\Ec{C_n(\xi
)}\sim\kappa n^{\beta}$, where $\kappa$ and $\beta$ are explicit
constants. We develop an approach based on the analysis of the cost
$C_n(s)$ of any fixed query $s\in[0,1]$, and give precise estimates
for the variance and limit distribution of the cost $C_n(x)$. Our
results permit us to describe a limit process for the costs $C_n(x)$ as
$x$ varies in $[0,1]$;
one of the consequences is that $\Ec{\max_{x\in[0,1]} C_n(x)} \sim
\gamma n^\beta$; this settles a question of Devroye [Pers. Comm., 2000].
\end{abstract}

%


\end{frontmatter}

\section{Introduction}\label{secintro}

Geometric databases arise in a number of contexts such as computer
graphics, management of geographical data or statistical analysis. The
aim consists in retrieving the data matching specified patterns
efficiently. We are interested in tree-like data structures which
permit such efficient searches. When the pattern specifies precisely
all the data fields (we are looking for an \emph{exact match}), the
query can generally be answered in time logarithmic in the size of the
database, and many precise analyses are available in this case, see,
for example, \cite
{FlLa1994,FlLaLaSa1995,Knuth1998,Mahmoud1992a,FlSe2009}. When the
pattern only constrains some of the data fields (we are looking for a
\emph{partial match}), the searches must explore multiple branches of
the data structure to report the matching data, and the cost usually
becomes polynomial.

The first investigations about partial match queries by Rivest \cite
{Rivest1976} were based on digital data structures (based on
bit-comparisons). In a comparison-based setting, where the data may be
compared directly at unit cost, a few general purpose data structures
generalizing binary search trees permit to answer partial match
queries, namely the quadtree \cite{FiBe1974}, the $k$-d tree \cite
{Bentley1975} and the relaxed $k$-d tree \cite{DuEsMa1998}.
Besides the interest that one might have in partial match for its own
sake, there are various reasons that justify the precise quantification
of the cost of such general search queries in comparison-based data
structures. First, these multidimensional trees are data structures of
choice for applications that range from collision detection in motion
planning to mesh generation \cite{YeSh1983a,hole1988}. Furthermore, the
cost of partial match queries also appears in (hence influences) the
complexity of a number of other geometrical search questions such as
range search \cite{DuMa2002a} or rank selection~\cite{DuJiMa2010}.
For general references on multidimensional data structures and more
details about their various applications, see the series of monographs
by Samet \cite{Samet1990a,Samet1990,Samet2006}.

In this paper, we provide refined analyses of the costs of partial
match queries in some of the most important two dimensional data
structures. We mostly focus on quadtrees. We extend our results to the
case of $k$-d trees in Section~\ref{seckd}. Similar results also hold
for relaxed $k$-d trees of Duch, Estivill-Castro, and Mart\'inez \cite{DuEsMa1998}. However, even stating
them carefully would require much space without shedding anymore light
on the phenomena, and we leave the straightforward modifications to the
interested reader.

\textsc{Quadtrees and multidimensional search.}
The quadtree \cite{FiBe1974} allows to manage multidimensional data by
extending the divide-and-conquer approach of the binary search tree.
Consider the point sequence $p_1,p_2,\ldots, p_n\in[0,1]^2$. As we
build the tree, regions of the unit square are associated to the nodes
where the points are stored. Initially, the root is associated with the
region $[0,1]^2$, and the data structure is empty. The first point
$p_1$ is stored at the root, and divides the unit square into four
regions, $Q_1, \ldots, Q_4$. Each region is assigned to a child of the
root. More generally, when $i$ points have already been inserted, we
have a set of $1+3i$ (lower-level) regions that cover the unit square.
The point $p_{i+1}$ is stored in the node (say $u$) that corresponds to
the region it falls in, and divides it into four new regions that are
assigned to the children of $u$. See Figure~\ref{figquadtree}.

%
%

\begin{figure}

\includegraphics{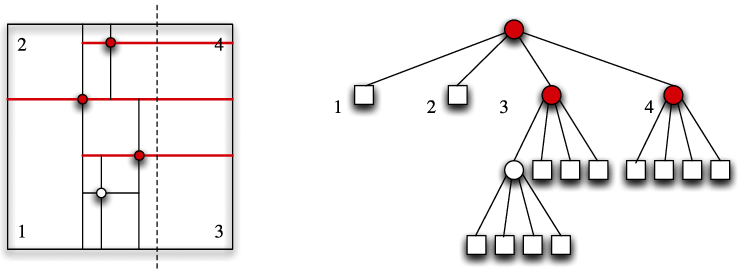}

\caption{An example of a (point) quadtree:
on the left the partition of the unit square induced by the tree data
structure on the right (the children are ordered according to the
numbering of the regions on the left). Answering the partial match
query materialized by the dashed line on the left requires one to visit
the colored nodes. Note that each one of the visited
nodes correspond to a horizontal line that is crossed by the query.}\label{figquadtree}
\end{figure}

\textsc{Analysis of partial match retrieval.}
For the analysis, we will focus on the model of \emph{random
quadtrees}, where the data points are independent and uniformly
distributed in the unit square. In the present case, the data are just
points and the problem of partial match retrieval consists in reporting
all the data with one of the coordinates (say the first) being $s\in[0,1]$.
It is a simple observation that the number of nodes of the tree visited
when performing the search is precisely $C_n(s)$, the number of
regions in the quadtree that intersect a vertical line at $s$. The
first analysis of partial match in quadtrees is due to Flajolet et~al. \cite
{FlGoPuRo1993} (after the pioneering work of Flajolet and Puech \cite{FlPu1986} in the
case of $k$-d trees). They studied the singularities of a differential
system for the generating
functions of partial match cost to prove that, for a random query $\xi
$, being independent of the tree and uniformly distributed on $[0,1]$,
one has
$ \Ec{C_n(\xi)}\sim\kappa  n^{\beta}$
where
%
\begin{equation}
\label{defbeta} \kappa= \frac{\Gamma(2\beta+2)}{2\Gamma(\beta+1)^3} \quad\mbox {and}\quad \beta=
\frac{\sqrt{17}-3}2,
\end{equation}
and $\Gamma(x)$ denotes the Gamma function $\Gamma(x)=\int_0^\infty
t^{x-1} e^{-t} \,dt$. Flajolet et~al. \cite{FlGoPuRo1993} actually proved a more precise
version of this estimate which will be crucial for us,
%
\begin{equation}
\label{hwangquad}\mathbf{E} \bigl[C_n(\xi) \bigr] = \kappa n^{\beta} -
1 + O \bigl(n^{\beta-1} \bigr).
\end{equation}
(This may also be obtained from the explicit expression for $\Ec{C_n(\xi
)}$ devised by Chern and Hwang~\cite{ChHw2003}.)

Our aim in this paper is to gain a refined understanding of the cost
beyond the level of expectations. In order to quantify the order of
typical deviations from the mean, we study the order of the variance
together with limit distributions. However, deriving higher moments
turns out to be subtle. In particular, when the query line is random
(like above) although the four subtrees at the root are independent
given their sizes, the contributions of the two subtrees that \emph{do
hit} the query line are \emph{dependent}. Indeed, the relative location
of the query line inside these two subtrees is again uniform, but
unfortunately it is same in both regions. Hence, one cannot easily
setup recurrence relations and perform an asymptotic analysis
exploiting independence. This issue has not yet been addressed
appropriately, and there is currently no result on the variance or
higher moments for~$C_n(\xi)$.

Another issue lies in the definition of the cost measure itself: even
if the data follow some distribution, should one assume that the query
follows the same distribution? In other words, should we focus on
$C_n(\xi)$? Maybe not. But then, what distribution should one use for
the query line?

One possible approach to overcome both problems is to consider the
query line to be fixed and to study $C_n(s)$ for $s\in[0,1]$. This
raises another problem: even if~$s$ is fixed at the top level, as the
search is performed, the relative location of the queries in the
recursive calls varies from one node to another. Thus, in following
this approach, one is led to consider the entire stochastic process
$(C_n(s))_{s\in[0,1]}$; this is the method we use here.

Recently Curien and Joseph \cite{CuJo2010} obtained some results in this direction. They
proved that for every fixed $s\in(0,1)$,
%
\begin{equation}
\label{constCJ} \mathbf{E}\bigl[C_n(s) \bigr]\sim K_1 \cdot h(s)
n^{\beta} \qquad\mbox{with } K_1 = \frac{\Gamma(2 \beta+2) \Gamma(\beta+ 2)}{2 \Gamma(\beta+ 1)^3
\Gamma (\beta/2 + 1  )^2},
\end{equation}
where the function $h$ defined below will play a central role in the
entire study
%
\begin{equation}
\label{defh} h(s):= \bigl(s(1-s) \bigr)^{\beta/2}.
\end{equation}
On the other hand, Flajolet et~al. \cite{FlGoPuRo1993,FlLaLaSa1995} prove that, along
the edge one has $\Ec{C_n(0)}= \Theta(n^{\sqrt{2}-1})$, so that $\Ec
{C_n(0)}=o(n^{\beta})$ (see also \cite{CuJo2010}). The behavior about
the $x$-coordinate $U$ of the first data point certainly resembles that
along the edge, so that one has $\Ec{C_n(U)}=o(n^{\beta})$. This
suggests that $C_n(s)$ should not be concentrated around its mean, %
and that $n^{-\beta}C_n(s)$ should converge to a nondegenerate random
variable as $n\to\infty$. Below, we confirm this and prove a functional
limit law for $(n^{-\beta}C_n(s))_{s\in[0,1]}$ and characterize the
limit process. From this we obtain refined asymptotic information on
the complexity of partial match queries in quadtrees.

\section{Main results and implications}\label{secmainresults}
We denote by $\Do$ the space of c\`adl\`ag functions on $[0,1]$ and by
$\|f\|:=\sup_{t\in[0,1]}|f(t)|$ the uniform norm of $f \in\Do$.
Our main contribution is to prove the following convergence result:
%
\begin{thmm}\label{thmprocess}Let $C_n(s)$ be the cost of a partial
match query at a fixed line~$s$ in a random quadtree. Then there exists
a random continuous function $Z$ such that, as $n\to\infty$,
%
\begin{equation}
\label{eqprocess} \biggl(\frac{C_n(s)}{K_1 n^{\beta}},s\in[0,1] \biggr) \stackrel{d}{\rightarrow}
\bigl(Z(s), s \in[0,1] \bigr).
\end{equation}
This convergence in distribution holds in $\Do$ equipped with the
Skorokhod topology.
\end{thmm}

\begin{figure}

\includegraphics{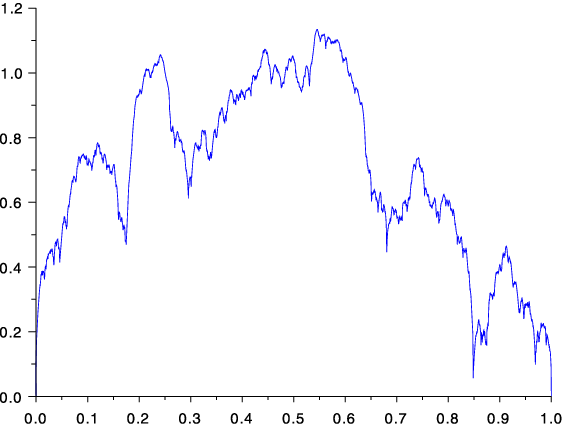}

\caption{A simulation of the limit process $Z$.}\label{figlimitprocess}
\end{figure}

The limit process $Z$ may be characterized as follows (see Figure~\ref
{figlimitprocess} for a simulation):
%
\begin{prop} \label{propchar}
The distribution of the random function $Z$ in \eqref{eqprocess} is a
fixed point of the following functional recursive distributional
equation, as process in $s\in[0,1]$:
%
\begin{eqnarray}
\label{eqlimitprocessa} Z(s)&\eqdist &\I{s<U} \biggl[(UV)^\beta
Z^{(1)} \biggl(\frac{ s}{ U} \biggr) + \bigl(U(1-V)
\bigr)^\beta Z^{(2)} \biggl(\frac{ s}{ U} \biggr) \biggr]
\nonumber
\\
&&{}+\I{s\ge U} \biggl[ \bigl((1-U)V \bigr)^\beta
Z^{(3)} \biggl( \frac{s-U}{1-U}\biggr)\\
&&\hspace*{14pt}\qquad\quad{}+ \bigl((1-U) (1-V)
\bigr)^\beta Z^{(4)}\biggl( \frac
{s-U}{1-U}\biggr) \biggr],\nonumber
\end{eqnarray}
where $U$ and $V$ are independent $[0,1]$-uniform random variables and
$Z^{(i)}$, $i=1,\ldots, 4$ are independent copies of the process $Z$,
which are also independent of $U$ and $V$. Furthermore, $Z$ in \eqref
{eqprocess} is the only continuous solution of \eqref
{eqlimitprocessa} with $\Ec{\|Z\|^2}<\infty$ and $\Ec{Z(\xi)} =
\Gamma(\beta/2+1)^2/\Gamma(\beta+ 2)$ where $\xi$ is independent of
$Z$ and uniformly distributed on $[0,1]$.
\end{prop}

The methods applied to prove Theorem~\ref{thmprocess} also
guarantee convergence of the variance of the costs of partial match
queries. The following theorem for uniform queries $\xi$ is the direct
extension of the pioneering work in \cite{FlPu1986,FlGoPuRo1993} for
the cost of partial match queries at a uniform line $\xi$ in random
two-dimensional trees.

\begin{thmm} \label{thmmoments}
If $\xi$ is uniformly distributed on $[0,1]$, independent of $(C_n)$
and~$Z$, then
\[
\frac{C_n(\xi)}{K_1 n^{\beta}} \to Z(\xi)
\]
in distribution, as $n\to\infty$. Moreover, $\V{C_n(\xi)} \sim K_4 n^{2
\beta}$
where $K_4\approx\break 0.447363034$ is given by, with $K_1$ in (\ref{constCJ}),
%
\begin{eqnarray}
\label{eqvarzxi} K_4 &:= &K_1^2 \cdot
\operatorname{Var}\bigl(Z( \xi)\bigr) 
\nonumber
\\[-8pt]
\\[-8pt]
\nonumber
&=&K_1^2
\biggl(\frac{2(2\beta+1)}{3(1-\beta)} \mathrm{B}(\beta+1,\beta+1)^2 - \mathrm{B}(
\beta/2+1, \beta/2+1)^2 \biggr).
\end{eqnarray}
\end{thmm}
Here $\mathrm{B}(a,b):=\int_0^1 t^{a-1}(1-t)^{b-1}\,dt$ denotes the
Eulerian integral for $a,b>-1$.
In particular, Theorem~\ref{thmmoments} identifies the first-order
asymptotics of\break $\Vc{C_n(\xi)}$ which is to be compared with studies
that neglected the dependence between the contributions of the subtrees
mentioned above \cite{MaPaPr2001,NeRu2001,Neininger2000}. A refined
result about the variance $\Vc{C_n(s)}$ at a fixed location reads
\[
\operatorname{Var} \bigl(C_n(s) \bigr) \sim K_1^2
\operatorname{Var} \bigl(Z(s) \bigr) n^{2 \beta},
\]
where $s\in(0,1)$ and an explicit expression for $\V{Z(s)}$ is given by
%
\begin{equation}
\label{constvarfix} \operatorname{Var} \bigl(Z(s) \bigr) =K_2
h^2(s) = \biggl[ 2 \mathrm{B} (\beta+1, \beta+1 ) \frac{2\beta+1}{3(1-\beta)}
-1 \biggr] h^2(s).
\end{equation}

Another consequence of Theorem \ref{thmprocess} concerns the order of
the cost of the worst query given by $\sup_{s\in[0,1]} C_n(s)$.
%
\begin{thmm}\label{thmsupremum}Let $S_n=\sup_{s\in[0,1]} C_n(s)$. Then
as $n\to\infty$,
\[
\frac{S_n}{K_1 n^{\beta}} \to S:=\sup_{s\in[0,1]} Z(s),
\]
in distribution and with convergence of all moments. In particular, $\Ec
{S}<\infty$, $\Vc{S}<\infty$ and we have
\[
\Ec{S_n} \sim K_1 n^{\beta} \Ec{S}\quad \mbox{and}\quad
\Vc{S_n} \sim K_1^2 n^{2\beta} \Vc{S}.
\]
\end{thmm}
Note that the sequence $n^{-\beta}\Ec{S_n}$ is bounded. In particular,
$\Ec{S_n}$ has the same order of magnitude as the cost of a search
query at any single location, and does not include any extra factor
growing with $n$. Interestingly, the one-dimensional marginals of the
limit process $(Z(s), s\in[0,1])$ are all the same up to a
deterministic multiplicative constant given by the function $h$:
%
\begin{thmm}\label{thmsame-Xi}
There exists a random variable $\Psi\ge0$ such that for all $s\in[0,1]$,
%
\begin{equation}
\label{1dimrv} Z(s)\stackrel{d} {=} h(s) \cdot\Psi.
\end{equation}
The distribution of $\Psi$ is the unique solution of the fixed-point equation
%
\begin{equation}
\label{fixedone} \Psi\stackrel{d} {=} U^{\beta/2} V^\beta\Psi+
U^{\beta/2} (1-V)^\beta \Psi
\end{equation}
with $\Ec{\Psi} = 1$ and $\Ec{\Psi^2} < \infty$ where $\Psi'$ is an
independent copy of $\Psi$ and $(\Psi,\Psi')$ is independent of $(U,V)$.
\end{thmm}
Convergence of all moments of the supremum $n^{-\beta}S_n$ in Theorem
\ref{thmsupremum} implies uniform integrability of any moment of the
process $n^{-\beta}C_n$, hence the following result about convergence
of all moments.
%
\begin{cor}\label{thmmixedn}
For all $s \in[0,1]$, we have
\[
\mathbf{E} \biggl[ \biggl(\frac{C_n(s)}{K_1 n^{\beta}} \biggr)^m \biggr]
\rightarrow \mathbf{E} \bigl[Z(s)^m \bigr] = c_m
h(s)^m
\]
for all $m \in\N$ as $n \to\infty$ where $c_m$ is given by
%
\begin{eqnarray}
\label{recmomentsZ} c_m &=& \frac{\beta m +1}{(m-1)  (m+1 - (3 /2) \beta m  )}
\nonumber
\\[-8pt]
\\[-8pt]
\nonumber
&&{}\times\sum
_{\ell= 1}^{m-1} \pmatrix{m \cr\ell} \mathrm{B} \bigl(\beta
\ell+1, \beta(m-\ell)+1 \bigr) c_\ell c_{m-\ell}
\end{eqnarray}
for $m \geq2$ where $c_1 = 1$.
An analogous result holds true for $\Ec{C_n(\xi)}$ where $\xi$ is
uniform on $[0,1]$ and independent of $(C_n)_{n \geq0}$ and $Z$, and
for moments involving queries at multiple locations.
\end{cor}

\textsc{Plan of the paper.} Our approach requires to work with
the process $(C_n(s)\dvtx s\in[0,1])$ and is based on the recursive
decomposition of the tree at the root. This yields a recursive
distributional recurrence for $(C_n(s)\dvtx s\in[0,1])$ to which we apply a
functional version of the contraction method. In Section~\ref
{seccontraction}, we give an overview of this underlying methodology.
In particular, we discuss the novel results of Neininger and Sulzbach
\cite{NeSu2011a} about the contraction method in function spaces which
we will apply.
Sections~\ref{seclimit} and~\ref{secunifconvergence} are dedicated
to the proofs of two of the main ingredients required to apply the
results from \cite{NeSu2011a}, the existence of a continuous solution
of the limit recursive equation and the uniform convergence of the
rescaled first moment $n^{-\beta}\Ec{C_n(s)}$ at an appropriate rate.
In Section~\ref{secmomentssup}, we identify the variance and the
supremum of the limit process $Z$ and deduce the large $n$ asymptotics
for $C_n(s)$ in Theorems~\ref{thmmoments} and~\ref{thmsupremum}.
Finally, we prove analogous results for the cases of $2$-d trees in
Section~\ref{seckd}. Our results on quadtrees have been announced in
the extended abstract \cite{BrNeSu2011a}.

\section{Contraction method in function spaces} \label{seccontraction}

\subsection{Overview of the method}\label{secoverviewcont}
The aim of this section is to give an overview of the method we employ
to prove Theorem~\ref{thmprocess}. It is based on a contraction
argument in a certain space of probability distributions. In the
context of the analysis of algorithms, the method was first employed by
R{\"o}sler \cite{Roesler1991a} who proved convergence in distribution for the
rescaled total cost of the randomized version of quicksort. The method
was then further developed by R{\"o}sler \cite{Roesler1992}, Rachev and R{\"u}schendorf~\cite{RaRu1995},
and later on in \cite{Ro01,Ne01,NeRu04,NeRu04b,DrJaNe,EiRu07} and
has permitted numerous analyses in distribution for random discrete structures.

So far, the method has mostly been used to analyze random variables
taking real values, though a few applications on function spaces have
been made; see \cite{DrJaNe,EiRu07,GR09}.
Here we are interested in the function space ${\cal D}[0,1]$ endowed
with the Skorokhod topology (see, e.g., \cite{Billingsley1999}), but
the main idea persists: (1) devise a recursive equation for the
quantity of interest [here the process $(C_n(s), s\in[0,1])$], and (2)~based on a properly rescaled version of the quantity deduce a limit
equation, that is, a recursive distributional equation that the limit
may satisfy; (3) if the map of distributions associated to the limit
equation is a contraction in a certain metric space, then a fixed point
is unique and may be obtained by iteration. The contraction may also be
exploited to obtain weak convergence to the fixed point. We now move on
to the first step of this program.

Write $I_1^{(n)},\ldots, I_4^{(n)}$ for the number of points falling in
the four regions created by the point stored at the root.
Then, given the coordinates of the first data point $(U,V)$, we have
(cf.~Figure~\ref{figquadtree})
%
\begin{eqnarray}
\label{eqsizes} &&\bigl(I_1^{(n)},\ldots,
I_4^{(n)} \bigr)
\nonumber
\\[-8pt]
\\[-8pt]
\nonumber
&&\qquad \eqdist\operatorname{Mult} \bigl(n-1;
UV,U(1-V), (1-U) (1-V), (1-U)V \bigr).
\end{eqnarray}
Observe that, for the cost inside a subregion, what matters is the
location of the query line \emph{relative} to the region. Thus a
decomposition at the root yields the following recursive relation for
any $n\ge1$:
%
\begin{eqnarray}
\label{eqCnrec} C_n(s)& \eqdist&1 + \I{s<U} \biggl[C^{(1)}_{I_1^{(n)}}
\biggl(\frac {s} {U}\biggr) + C^{(2)}_{I_2^{(n)}} \biggl({
\frac{ s}{ U}}\biggr)\biggr]
\nonumber
\\[-8pt]
\\[-8pt]
\nonumber
&&{}+\I{s\ge U} \biggl[C^{(3)}_{I_3^{(n)}}\biggl(\frac{1-s}{1-U}
\biggr) + C^{(4)}_{I_4^{(n)}} \biggl(\frac{1-s}{1-U}\biggr)
\biggr],
\end{eqnarray}
where $U,I_1^{(n)},\ldots, I_4^{(n)}$ are the quantities already
introduced and $(C^{(1)}_k), \ldots,\break (C^{(4)}_k)$ are independent copies
of the sequence $(C_k, k\ge0)$, independent of $(U,V,I_1^{(n)}, \ldots,
I_4^{(n)})$. We stress that this equation does not only hold true
pointwise for fixed $s$ but also as c\`adl\`ag functions on the unit
interval. The relation in \eqref{eqCnrec} is the fundamental equation
for us.

Letting $n\to\infty$ (formally) in \eqref{eqCnrec} suggests that if
$n^{-\beta} C_n(s)$ does converge to a random variable $Z(s)$ in a
sense to be made precise, then the distribution of the process $(Z(s),
0\le s\le1)$ should satisfy the following fixed point equation:
%
\begin{eqnarray}
\label{eqlimitprocess} Z(s)&\eqdist &\I{s<U}\biggl[(UV)^\beta
Z^{(1)}\biggl(\frac {s}{ U}\biggr)+ \bigl(U(1-V)
\bigr)^\beta Z^{(2)}\biggl(\frac {s}{ U}\biggr)\biggr]\nonumber
\\
&&{}+\I{s\ge U}\biggl[ \bigl((1-U)V \bigr)^\beta Z^{(3)} \biggl(
\frac{s-U}{1-U}\biggr)\\
&&\hspace*{25pt}\qquad{}+ \bigl((1-U) (1-V) \bigr)^\beta
Z^{(4)}\biggl( \frac{s-U}{1-U}\biggr)\biggr],\nonumber
\end{eqnarray}
where $U$ and $V$ are independent $[0,1]$-uniform random variables and
$Z^{(i)}$, $i=1,\ldots, 4$ are independent copies of the process $Z$,
which are also independent of $U$ and $V$.

The last step leading to the fixed point equation \eqref
{eqlimitprocess} needs now to be made rigorous. It is at this point
that the contraction method enters the game. The distribution of a
solution to our fixed-point equation (\ref{eqlimitprocess}) lies in
the set of probability measures on the Polish space $(\mathcal D[0,1],
d)$, which is the set we have to endow with a suitable metric. Here,
$d$ denotes the Skorokhod metric; see, for example, \cite{Billingsley1999}.

The recursive equation (\ref{eqCnrec}) is an example for the
following, more general setting of random additive recurrences: Let
$(X_n)$ be $\Do$-valued random variables with
%
\begin{equation}
\label{eqrecgen} X_n \stackrel{d} {=} \sum
_{r=1}^K A_r^{(n)} \bigl(
X_{I_r^{(n)}}^{(r)} \bigr) + b^{(n)},\qquad n \geq1,
\end{equation}
where $(A_1^{(n)}, \ldots, A_K^{(n)})$ are random continuous linear
operators on $\Do$, $b^{(n)}$ is a $\Do$-valued random variable,
$I_1^{(n)}, \ldots, I_K^{(n)}$ are random integers between $0$ and
$n-1$ and the sequences of process $(X_n^{(1)})$, \ldots, $(X_n^{(K)})$
are distributed like $(X_n)$. Moreover $(A_1^{(n)}, \ldots, A_K^{(n)},
b^{(n)}, I_1^{(n)}, \ldots, I_K^{(n)})$,\break $(X_n^{(1)}), \ldots,
(X_n^{(K)})$ are independent.

At this point, one should comment on the term random continuous linear
operator: As explained explicitly in \cite{NeSu2011a}, $A$ is a random
continuous linear operator
on $\Do$, if it takes values in the set of endomorphisms on $\Do$ that
are both continuous with respect to the supremum norm and to the
Skorokhod metric. Moreover,
for any $f \in\Do$ and $t \in[0,1]$, the quantity $Af(t)$ has to be a
real-valued random variable, and the same is assumed for $\|A\|_\mathrm
{op}$ (see below for the definition).
Finally, we remember
that convergence $d(f_n,f) \rightarrow0$ in the Skorokhod\vadjust{\goodbreak} metric means
that there exists a sequence of monotonically increasing bijections
$(\lambda_n)$ on the
unit interval such that $f_n(\lambda_n(t)) \rightarrow f(t)$ and
$\lambda_n(t) \rightarrow t$ both uniformly in $t$ as $n \rightarrow
\infty$.

To establish Theorem \ref{thmprocess} as a special case of this
setting, we use Proposition \ref{propcont} below. Proposition~\ref
{propcont} is part of the main convergence theorem in Neininger and Sulzbach \cite
{NeSu2011a}. We first state conditions needed to deal with the general
recurrence~(\ref{eqrecgen}); we will then justify that it can indeed
be used in the case of cost of partial match queries. Consider the
following assumptions, where, for a random variable $X$ in $\Do$ we
write $\|X\|_2:= \Ec{\|X\|^2}^{1/2}$, for a linear operator $A$ we
write $\|A\|_2:=\Ec{\|A\|_\mathrm{op}^2}^{1/2}$ with $\| A\|_\mathrm
{op}:= \sup_{\|x\|=1} \|A(x)\|$. Suppose $(X_n)$ obeys~(\ref
{eqrecgen}) and the following:
\begin{longlist}[(A3)]
\item[(A1)] \textsc{Convergence and contraction.}
We have $\|A_r^{(n)}\|_2, \|b^{(n)}\|_2<\infty$ for all $r=1,\ldots,K$
and $n\ge0$ and there exist random continuous linear operators $A_1,
\ldots, A_K$ on $\Do$ and a $\Do$-valued random variable $b$ such that,
for some positive sequence $R(n) \downarrow0$, as $n\to\infty$,
%
\begin{equation}
\label{eqrate1} \bigl\|b^{(n)} - b\bigr\|_2+\sum
_{r=1}^K \bigl\|A_r^{(n)} -
A_r\bigr\|_2= O \bigl(R(n) \bigr)
\end{equation}
and for all $\ell\in\N$,
\[
\mathbf{E} \bigl[\mathbf{1}_{\{I^{(n)}_r \in\{0,\ldots,\ell\}\}} \bigl\|A^{(n)}_r
\bigr\|_\mathrm {op}^2 \bigr] \to0
\]
and
%
\begin{equation}
L^{*} = \limsup_{n \rightarrow\infty} \mathbf{E} \Biggl[\sum
_{r=1}^K\bigl \|A_r^{(n)}
\bigr\|_\mathrm{op}^2 \frac
{R(I_r^{(n)})}{R(n)} \Biggr] < 1.
\end{equation}

\item[(A2)] \textsc{Existence and equality of moments.} $\Ec{\|X_n\|
^2} < \infty$ for all $n$ and $\Ec{X_{n_1}(t)} = \Ec{X_{n_2}(t)}$ for
all $n_1,n_2 \in\N_0,
t\in[0,1]$.
\item[(A3)] \textsc{Existence of a continuous solution.} There
exists a solution $X$ of the fixed-point equation
%
\begin{equation}
\label{fix} X \stackrel{d} {=} \sum_{r=1}^K
A_r \bigl(X^{(r)} \bigr) + b
\end{equation}
with continuous paths, $\Ec{\|X\|^2} < \infty$ and $\Ec{X(t)} = \Ec
{X_1(t)}$ for all $t\in[0,1]$.
Again the random variables $(A_1, \ldots, A_K,b), X^{(1)}, \ldots,
X^{(K)}$ are independent and $X^{(1)}, \ldots, X^{(K)}$ are distributed
like $X$.
\item[(A4)] \textsc{Perturbation condition.} $X_n = W_n + h_n$
where $ \|h_n - h \| \to0$ with $h\in\Co$ and random variables $W_n$
in $\Do$ such that there exists a sequence $(r_n)$ with, as $n\to\infty$,
\[
\mathbf{P} \bigl(W_n \notin\Drn \bigr) \to0.
\]
Here, $\Drn\subset\Do$ denotes the set of functions on the unit
interval continuous at $1$, for which there is a decomposition of
$[0,1]$ into intervals of length as least $r_n$ on which they are
constant.\vadjust{\goodbreak}
\item[(A5)] \textsc{Rate of convergence.} $R(n) =o (\log
^{-2}(1/r_n) )$.
\end{longlist}

The contraction method presented here for the space $(\Do,d)$ is based
on the Zolotarev metric $\zeta_2$; see \cite{NeSu2011a}. We state the
part of the main
convergence theorem of Neininger and Sulzbach \cite{NeSu2011a} that we will use. In the next
section, we will prove our main result, Theorem~\ref{thmprocess}, with
the help of
Proposition~\ref{propcont}.

\begin{prop} \label{propcont}
Let $(X_n)$ fulfill (\ref{eqrecgen}). Provided that assumptions
\textup{(A1)--(A3)} are satisfied, the solution $X$ of the fixed-point equation
(\ref{fix}) is unique.
\begin{longlist}[(iii)]
\item[(i)] For all $t\in[0,1]$, $X_n(t) \to X(t)$ in distribution, with
convergence of the first two moments.
\item[(ii)] If $\xi$ is independent of $(X_n), X$ and distributed on
$[0,1]$, then $X_n(\xi)\to X(\xi)$ in distribution again with
convergence of the first two moments.
\item[(iii)] If also \textup{(A4)} and \textup{(A5)} hold, then $X_n \rightarrow X$ in
distribution in $(\Do, d)$.
\end{longlist}
\end{prop}
Note that $X_n \rightarrow X$ in distribution in $(\Do, d)$ with $X$
having continuous sample paths implies that we can find versions of
$(X_n),X$ on a suitable probability space such that
$\| X_n - X\| \rightarrow0$ almost surely. However, in general we do
not have $X_n \rightarrow X$ in distribution in $\Do$ endowed with the
uniform topology due to problems with
measurability; see \cite{Billingsley1999}, Section 15 and \cite{NeSu2011a}, Section
2.2.

\subsection{\texorpdfstring{The functional limit theorem: Proof of Theorem~\protect\ref{thmprocess}}{The functional limit theorem: Proof of Theorem 1}}
The aim of this section is to prove Theorem~\ref{thmprocess} with the
help of Proposition~\ref{propcont} from Neininger and Sulzbach~\cite{NeSu2011a}. More
precisely, in the following we prove conditions (A1)--(A5), except two
which require much more work: the existence of a continuous solution~(A3), and the uniform convergence of the mean in (A1) are treated
separately in Sections~\ref{seclimit} and~\ref{secunifconvergence},
respectively.

Following the heuristics in the \hyperref[secintro]{Introduction} we scale the additive
recurrence~(\ref{eqCnrec}) by $n^{\beta}$. Let $Q_0(t):=0$ and
\[
Q_n(t) = \frac{C_n(t)} {K_1 n^\beta},\qquad n\ge1.
\]
The recursive distributional equation then rewrites in terms of $Q_n$ as
%
\begin{eqnarray}
\label{eqscalrec} &&\bigl( Q_n(t) \bigr)_{t \in[0,1]}\nonumber\\
&&\qquad \stackrel{{d}}
{=} \biggl( \I{t < U} \biggl[ \biggl(\frac{I_1^{(n)}}{n} \biggr)^{\beta}
Q_{I_1^{(n)}}^{(1)} \biggl( \frac{t}{U} \biggr) + \biggl(
\frac{I_2^{(n)}}{n} \biggr)^{\beta} Q_{I_2^{(n)}}^{(2)} \biggl(
\frac{t}{U} \biggr) \biggr]
\nonumber
\\[-8pt]
\\[-8pt]
\nonumber
&&\quad\qquad\hspace*{6pt}{} + \I{t \geq U} \biggl[ \biggl(\frac{I_3^{(n)}}{n} \biggr)^{\beta}
Q_{I_3^{(n)}}^{(3)} \biggl( \frac{t-U}{1-U} \biggr) + \biggl(
\frac{I_4^{(n)}}{n} \biggr)^{\beta} Q_{I_4^{(n)}}^{(4)} \biggl(
\frac
{t-U}{1-U} \biggr) \biggr]
\\
&&\hspace*{188pt}\quad\qquad\hspace*{30pt}\qquad{} + \frac{1}{K_1 n^{\beta}} \biggr)_{t \in[0,1]},\nonumber
\end{eqnarray}
where $U,I_1^{(n)},\ldots, I_4^{(n)}$ are the quantities already
introduced in Section~\ref{secoverviewcont} and~\eqref{eqsizes} and
$(Q^{(1)}_n)_{n \geq0}, \ldots, (Q^{(4)}_n)_{n \geq0}$ are
independent copies\vspace*{1pt} of $(Q_n)_{n \geq0}$,
independent of $(U,V,I_1^{(n)},\ldots, I_4^{(n)})$. The convergence of
the coefficients $(I_j^{(n)}/n )^{\beta}$ suggests
that a limit of $Q_n(t)$ should satisfy the fixed-point equation (\ref
{eqlimitprocess}).

\textsc{The recurrence relation.}
Most details consist in setting the right form of the recurrence
relation: for (A2) to be satisfied, we need to use a scaling that leads
to an expectation which is independent of $n$. This is not the case for
$Q_n(t)$. Denoting
$\mu_n(t) = \E{C_n(t)}$, we are naturally led to consider $Y_0(t):=0$ and
\[
Y_n(t) = \frac{C_{n}(t) - \mu_n(t)}{K_1 n^{\beta}} = Q_n(t) - h(t) + O
\bigl(n^{-\varepsilon} \bigr),\qquad n\ge1,
\]
where the error term is deterministic and uniform in $t \in[0,1]$.
Hence it is sufficient to prove convergence of the sequence
$(Y_n)_{n\ge1}$. The distributional recursion in terms of $Y_n$ is
\begin{eqnarray*}
&&\bigl( Y_n(t) \bigr)_{t \in[0,1]}\\
&&\qquad \stackrel{{d}} {=} \biggl( \I{t
< U} \biggl[ \biggl(\frac{I_1^{(n)}}{n} \biggr)^{\beta}
Y_{I_1^{(n)}}^{(1)} \biggl( \frac{t}{U} \biggr) + \biggl(
\frac
{I_2^{(n)}}{n} \biggr)^{\beta} Y_{I_2^{(n)}}^{(2)} \biggl(
\frac
{t}{U} \biggr) \biggr]
\\
&&\qquad\quad\hspace*{6pt}{} + \I{t \geq U} \biggl[ \biggl(\frac{I_3^{(n)}}{n} \biggr)^{\beta}
Y_{I_3^{(n)}}^{(3)} \biggl( \frac{t-U}{1-U} \biggr) + \biggl(
\frac
{I_4^{(n)}}{n} \biggr)^{\beta} Y_{I_4^{(n)}}^{(4)} \biggl(
\frac
{t-U}{1-U} \biggr) \biggr]
\\
&&\qquad\quad\hspace*{6pt}{} + \I{t < U} \biggl[\frac{\mu_{I_1^{(n)}} ({t}/{U} )+ \mu
_{I_2^{(n)}} ({t}/{U} )}{K_1 n^{\beta}} \biggr]
\\
&&\qquad\quad\hspace*{6pt}{} + \I{t \geq U} \biggl[ \frac{\mu_{I_3^{(n)}} ({(t-U)}/{(1-U)}
)+ \mu_{I_4^{(n)}} ({(t-U)}/{(1-U)} )}{K_1 n^{\beta}} \biggr] \\
&&\hspace*{223pt}\qquad\quad{}+ \frac{1 - \mu_n(t)}{K_1 n^{\beta}}
\biggr)_{t \in[0,1]},
\end{eqnarray*}
where $(Y_n^{(1)})_{n \geq0}, \ldots, (Y_n^{(4)})_{n \geq0}$ are
independent copies of $(Y_n)_{n \geq0}$ which are also independent of
the vector
$(U,V,I_1^{(n)},\ldots, I_4^{(n)})$.
Therefore, any possible limit $Y$ of $Y_n$ should satisfy the following
distributional fixed-point equation:
%
\begin{eqnarray}
\label{eqfix-mod} &&\hspace*{-5pt}\bigl( Y(t) \bigr)_{t \in[0,1]} \nonumber\\
&&\qquad\stackrel{{d}} {=} \biggl( \I{t <
U} \biggl[(UV)^{\beta} Y^{(1)} \biggl(\frac{t}{U} \biggr)
+ \bigl(U(1-V) \bigr)^{\beta} Y^{(2)} \biggl(\frac{t}{U}
\biggr) \biggr]
\nonumber
\\
&&\qquad\quad\hspace*{1pt}{} + \I{t \geq U} \biggl[ \bigl((1-U)V \bigr)^{\beta} Y^{(3)}
\biggl(\frac
{t-U}{1-U} \biggr)
\nonumber
\\[-8pt]
\\[-8pt]
\nonumber
&&\qquad\quad\hspace*{22pt}\qquad{}+ \bigl((1-U) (1-V) \bigr)^{\beta}
Y^{(4)} \biggl(\frac
{t-U}{1-U} \biggr) \biggr]
\\
&&\qquad\quad\hspace*{1pt}{} + \I{t \geq U} h \biggl(\frac{t-U}{1-U} \biggr) \bigl( \bigl((1-U)V
\bigr)^{\beta} + \bigl((1-U) (1-V) \bigr)^{\beta} \bigr) - h(t)
\nonumber
\\
&&\qquad\quad\hspace*{103pt}{}+ \I{t < U} h \biggl(\frac{t}{U} \biggr) \bigl( (UV)^{\beta} +
\bigl(U(1-V) \bigr)^{\beta} \bigr) \biggr)_{t \in[0,1]}.\nonumber
\end{eqnarray}
Having Proposition~\ref{propcont} in mind, we define (random)
operators $A_r^{(n)}$, $r=1,2,3,4$, by
\[
A_r^{(n)}(f) (t) = %
\cases{\displaystyle \I{t < U} \biggl(
\frac{I_r^{(n)}}{n} \biggr)^{\beta} f \biggl( \frac
{t}{U} \biggr), & \quad $\mbox{if } r = 1,2$,\vspace*{2pt}
\cr
\displaystyle\I{t \geq U} \biggl(\frac{I_r^{(n)}}{n}
\biggr)^{\beta} f \biggl( \frac
{t-U}{1-U} \biggr), & \quad $\mbox{if } r=
3,4.$} %
\]
Furthermore let $b^{(n)}(t) = \sum_{r=1}^4 b^{(n)}_r(t) + (1 - \mu
_n(t))/(K_1 n^{\beta})$ with
\[
b^{(n)}_r(t) = %
\cases{ \displaystyle\I{t < U} \cdot
\frac{\mu_{I_r^{(n)}} ({t}/{U} )
}{K_1n^{\beta}}, &\quad  $\mbox{if } r=1,2,$\vspace*{2pt}
\cr
\displaystyle\I{t \geq U} \cdot
\frac{\mu_{I_r^{(n)}} ({(t-U)}/{(1-U)} )
}{K_1n^{\beta}}, &\quad $\mbox{if } r= 3,4.$} %
\]
Then the finite-$n$ version of the recurrence relation for $(Y_n)_{n\ge
0}$ is precisely of the form of \eqref{eqrecgen}.

We define similarly the coefficients of the limit recursive equation
\eqref{eqfix-mod}. We will then show that with these definitions,
assumptions (A1)--(A5) are satisfied (again, except the existence of a
continuous limit solution and the uniform convergence for the mean
treated in Section~\ref{seclimit} and~\ref{secunifconvergence}). The
operators $A_1, \ldots, A_4$ are defined by
\begin{eqnarray*}
A_1(f) (t) &=& \I{t < U} (
UV )^{\beta} f \biggl( \frac{t}{U} \biggr),\\
 A_2(f) (t)
 &=&\I{t < U} \bigl( U(1-V) \bigr)^{\beta} f \biggl( \frac{t}{U}
\biggr),
\\
A_3(f) (t) & = &\I{t \geq U} \bigl( (1-U)V \bigr)^{\beta} f
\biggl( \frac{t-U}{1-U} \biggr), \\
 A_4(f) (t) &=& \I{t \geq U} \bigl(
(1-U) (1-V) \bigr)^{\beta} f \biggl( \frac{t}{U} \biggr)
\end{eqnarray*}
and $b(t) = \sum_{r=1}^4 b_r(t) - h(t)$ with
\begin{eqnarray*}
b_1(t) & =& \I{t < U} (
UV )^{\beta} h \biggl( \frac{t}{U} \biggr),\\
 b_2(t)& =&
\I{t < U} \bigl( U(1-V) \bigr)^{\beta} h \biggl( \frac{t}{U}
\biggr),
\\
b_3(t) & = &\I{t \geq U} \bigl( (1-U)V \bigr)^{\beta} h
\biggl( \frac
{t-U}{1-U} \biggr),\\
  b_4(t) &=& \I{t \geq U} \bigl(
(1-U) (1-V) \bigr)^{\beta} h \biggl( \frac{t}{U} \biggr).
\end{eqnarray*}
The operators $A_1, \ldots, A_4, A_1^{(n)}, \ldots, A_4^{(n)}$ are
linear for each $n$. Moreover, they are bounded above by
one, which implies that they are norm-continuous.
Their norm functions are real-valued random variables. In order to
establish that they are indeed random continuous
linear operators on $(\Do, d)$ it remains to check that they are
continuous with respect to the Skorokhod topology. To this end, it is sufficient
to prove that
\[
d(f_n,f) \rightarrow0\quad \Rightarrow\quad d \biggl(\I{t < u} f_n
\biggl(\frac tu \biggr),\I{t < u} f \biggl( \frac tu \biggr) \biggr) \rightarrow0
\]
for any $u \in[0,1]$.
This follows
easily since $\|f_n(\lambda_n(t))-f(t) \| \rightarrow0$ with
monotonically increasing bijections $\lambda_n$ on the unit interval
such that $\|\lambda_n(t) -t \| \rightarrow0$
implies $\| \I{\beta_n(t) < u} f_n(\beta_n(t)/u)-\I{t <u} f(t/u)\|
\rightarrow0$ where $\beta_n(t) = u   \lambda_n(t/u)$ for $t \leq u$
and $\beta_n(t) = t$ for $t > u$.

We are now ready to check that assumptions (A1)--(A5) indeed hold,
taking the results of Sections~\ref{seclimit} and~\ref
{secunifconvergence} for granted.

(A3) \textsc{Existence of a continuous solution.} In Section~\ref
{seclimit}, we construct a continuous solution $Z$ of the fixed-point
equation (\ref{eqlimitprocess}) with $\Ec{\|Z\|^2} < \infty$ and
$\Ec{Z(t)} = h(t)=(t(1-t))^{\beta/2}$. Hence the function $Y(t) = Z(t)
- h(t)$ is a continuous solution of (\ref{eqfix-mod}) with $\Ec{Y(t)}
= 0$ and
$\Ec{\|Y\|^2} < \infty$. A direct computation shows that $\Ec{\|A_r\|
_\mathrm{op}^2} = \Ec{(UV)^{2\beta}} = (2 \beta+1)^{-2}$, for
$r=1,\ldots, 4$. Observe that
\[
L:= \sum_{r=1}^4 \mathbf{E} \bigl[
\|A_r\|_\mathrm{op}^2\bigr] = \frac{4}{(2 \beta+1)^2}
< 1.
\]
In particular, $Y$ is the unique solution of (\ref{eqfix-mod}) with $\E
{Y(t)} = 0$ and $\Ec{\|Y\|^2} < \infty$. Thus, $Z$ is the unique
solution of \eqref{eqlimitprocessa} with $\Ec{Z(t)} = h(t)$ and $\Ec
{\|Z\|^2} < \infty$. By the arguments in \cite{CuJo2010}, Section 5,
the mean function of any process with c{\`a}dl{\`a}g paths and finite
moments satisfying \eqref{eqlimitprocessa} is a multiple of $h(s)$.
Hence, we may replace the condition $\E{Z(t)} = h(t)$ by $\E{Z(\xi)} =
\Gamma(\beta/2 + 1)^2/ \Gamma(\beta+2)$ as formulated in Proposition
\ref{propchar}.

(A2) \textsc{Existence and equality of moments.} The precise
scaling we chose ensures that $\Ec{Y_n(t)}=0$, for all $n\ge1$ and
$t\in[0,1]$. The second moments $\Ec{\|Y_n\|^2}$ are finite as the
random variables $\|Y_n\|$ are bounded for every fixed $n$.

(A1) \textsc{Convergence and contraction.} It suffices to focus on the terms
\[
\bigl\|A_1^{(n)}-A_1\bigr\|_2 \quad\mbox{and}\quad
\bigl\|b^{(n)}_1 - b_1\bigr\|_2,
\]
and the remaining terms can obviously be treated in the same way.
Establishing the convergence only boils down to verifying that a
binomial random variable Bin$(n,p)$ is properly approximated by $np$.
Using the Chernoff--Hoeffding inequality for binomials \cite
{Hoeffding1963}, one easily verifies that for every $\alpha>0$,
%
\begin{equation}
\label{eqalphamomentbin} \mathbf{E} \biggl[ \biggl\llvert \frac{\operatorname{Bin}(n,p)}{n} - p
\biggr\rrvert ^{\alpha} \biggr] = O \bigl(n^{-
\alpha/2} \bigr),
\end{equation}
uniformly in $p \in[0,1]$. Thus, since $|x^\beta-y^\beta|\le
|x-y|^\beta$ for any $x,y\in[0,1]$, we have
%
\begin{equation}
\label{eqbounderrA} \bigl\|A_1^{(n)} - A_1
\bigr\|_2 \leq \biggl\llVert \biggl(\frac{I_r^{(n)}}{n}
\biggr)^{\beta} - (UV)^{\beta} \biggr\rrVert _2 = O
\bigl( n^{-1/2} \bigr).
\end{equation}
By Proposition~\ref{PROPUNIFORMCONVERGENCE} we have $\mu_n(t) = K_1
h(t)n^{\beta} + O (n^{\beta- \varepsilon})$ uniformly in $t \in[0,1]$.
Therefore
\[
\bigl\| b_1^{(n)} - b_1 \bigr\|_2 \leq \biggl
\llVert \I{t < U} h \biggl(\frac{t}{U} \biggr) \biggl( \biggl(
\frac{I_r^{(n)}}{n} \biggr)^{\beta} - (UV)^{\beta} \biggr) \biggr
\rrVert _2 + C \biggl\llVert \frac{(I_1^{(n)})^{\beta- \varepsilon}}{n^{\beta}} \biggr\rrVert
_2
\]
for some constant $C > 0$. Since $h$ is bounded, the first summand is
$O( n^{-1/2})$ just like in \eqref{eqbounderrA} above.
The second term is trivially bounded by $C n^{-\varepsilon}$. Overall,
we have $\|b_1^{(n)} - b_1\|_2 = O(n^{-\varepsilon})$. Hence, since the
coefficients
$A_r^{(n)}$ are bounded by one in the operator norm and by
distributional properties of $I_1^{(n)}, \ldots, I_4^{(n)}$, the first
two constraints in assumption (A1) are satisfied with $R(n) = C
n^{-\varepsilon}$ for a suitable constant $C > 0$, and $\varepsilon>0$
may still be chosen as small as we want.

Next, we consider $L^*$ in (A1). By dominated convergence we have
\[
L^{*}  = 4 \mathbf{E} \bigl[(UV)^{2\beta} (UV)^{-\varepsilon}
\bigr] = \frac{4}{(2\beta-
\varepsilon+1)^2} < 1
\]
for $\varepsilon>0$ sufficiently small. This completes the verification
of (A1).

(A4) \textsc{Perturbation condition.} Note that $Q_n$ is
piecewise constant: $Q_n(t) = Q_n(s)$ for all $s,t$ if no
$x$-coordinate of the first $n$ points lies between $s$ and $t$. There
are $n$ independent points,
the probability that there exist two lying within $n^{-3}$ of each
other is at most $n^{-1}$. So (A4) is satisfied with $r_n = n^{-3}$.

(A5) \textsc{Rate of convergence.} With $r_n=n^{-3}$ and $R_n=C
n^{-\varepsilon}$, we have $R_n=o(\log^{-2} n)=o(\log^{-2}(1/ r_n))$.
Therefore, the condition on the rate of convergence is satisfied.

\section{The limit process}\label{seclimit}

In this section, we prove the existence of a process $Z\in\mathcal
C[0,1]$, the space of continuous functions from $[0,1]$ into $\R$, that
satisfies
the distributional fixed point equation \eqref{eqlimitprocess} and
whose mean matches the mean of the rescaled version $Y_n(s)$ of
$C_n(s)$. We construct the process $Z$ as the point-wise limit of
martingales. We then show that the convergence is actually almost
surely uniform, which allows us
to conclude that $Z\in\mathcal C[0,1]$ with probability one.
Figure~\ref{figlimitprocess} shows a simulation of the process $Z$.

We identify the nodes of the infinite quaternary tree with the set of
finite words on the alphabet $\{1,2,3,4\}$,
\[
\mathcal T = \bigcup_{n\ge0}\{1,2,3,4
\}^n.
\]
For a node $u\in\mathcal T$, we write $|u|$ for its depth, that is,
the distance between $u$ and the root
$\varnothing$. The descendants of $u\in\mathcal T$ correspond to all
the words in $\mathcal T$ with prefix~$u$; in particular, the children
of $u$ are $u1, \ldots, u4$. Let $\{U_v, v\in\mathcal T\}$
and $\{V_v, v\in\mathcal T\}$ be two independent families of i.i.d.
$[0,1]$-uniform random variables. By $\mathcal C_0[0,1]$ we denote the
set of continuous
functions on the unit interval vanishing at the boundary, that is,
$f(0) = f(1)= 0$ for $f \in\mathcal C_0[0,1].$
Define the continuous operator $G\dvtx (0,1)^2 \times\mathcal C_0[0,1]^4
\to\mathcal C_0[0,1]$ by
%
\begin{eqnarray}\label{eqdefG}
&& G(x,y,f_1,f_2,f_3,f_4) (s) \nonumber\\
&&\qquad=
\I{s<x} \biggl[(xy)^\beta f_1\biggl(\frac sx\biggr)+
\bigl(x(1-y) \bigr)^\beta f_2\biggl(\frac sx\biggr) \biggr]
\\
&&\qquad\quad{}+\I{s\ge x} \biggl[ \bigl((1-x)y \bigr)^\beta f_3\biggl(
\frac{s-x}{1-x}\biggr)+ \bigl((1-x) (1-y) \bigr)^\beta
f_4 \biggl(\frac{s-x}{1-x}\biggr) \biggr].\nonumber
\end{eqnarray}
Recall the definition of $h$ in \eqref{defh}.
For every node $u\in\mathcal T$, let $Z_0^{u}=h$. Then define
recursively
%
\begin{equation}
\label{defrecZ} Z_{n+1}^u=G \bigl(U_u,
V_u, Z_n^{u1}, Z_n^{u2},
Z_n^{u3}, Z_n^{u4} \bigr).
\end{equation}
Finally, define $Z_n=Z_n^\varnothing$ to be the value observed at the
root of $\mathcal T$ when the iteration has been started with $h$ in
all the nodes
at level $n$. We will see that for every $s\in[0,1]$, the sequence
$(Z_n(s), n\ge0)$ is a nonnegative discrete time martingale; so it
converges with
probability one to a finite limit.

It will be convenient to have an explicit representation for $Z_n$. For
$s\in[0,1]$, $Z_n(s)$ is the sum of exactly $2^n$ terms, each one
being the
contribution of one of the boxes at level $n$ that is cut by the line
at $s$. Let $\{Q_i^n(s), 1\le i\le2^n\}$ be the set of rectangles at
level $n$
whose first coordinate intersect $s$. Suppose that the projection of
$Q_i^n(s)$ on the first coordinate yields the interval $[\ell_i^n,
r_i^n]$. Then
%
\begin{equation}
\label{Znexplicit} Z_n(s)=\sum_{i=1}^{2^n}
\vol \bigl(Q_i^n(s) \bigr)^\beta\cdot h
\biggl( \frac{s-\ell
_i^n}{r_i^n-\ell_i^n}\biggr),
\end{equation}
where $\vol(Q_i^n(s))$ denotes the volume of the rectangle $Q_i^n(s)$.
The difference between $Z_n$ and $Z_{n+1}$ only relies in what happens
inside the boxes
$Q_i^n(s)$: We have
%
\begin{eqnarray}
\label{eqZn-telescoping} &&Z_{n+1}(s)-Z_n(s)
\nonumber
\\[-4pt]
\\[-12pt]
\nonumber
&&\qquad=\sum
_{i=1}^{2^n} \vol \bigl(Q_i^n(s)
\bigr)^\beta\cdot \biggl[G \bigl(U_i',V_i',h,h,h,h
\bigr) \biggl(\frac{s-\ell_i^n}{r_i^n-\ell_i^n}\biggr)-h\biggl(\frac{s-\ell_i^n}{r_i^n-\ell_i^n}\biggr)
\biggr],\hspace*{-35pt}
\end{eqnarray}
where $U'_i,V'_i$, $1\le i\le2^n$ are i.i.d. $[0,1]$-uniform random
variables. In fact, $U_i'$ and $V_i'$ are some of the variables $U_u,
V_u$ for nodes $u$ at
level $n$. Observe that, although the area $\vol(Q_i^n(s))$ is \emph
{not} a product of $n$ independent terms of the form $U V$ because of
size-biasing, but $U'_i, V'_i$ are in
fact \emph{unbiased}, that is, uniform. Let $\mathscr F_n$ denote the
$\sigma$-algebra generated by $\{U_u, V_u\dvtx |u|< n\}$.
Then the family $\{U_i',V_i'\dvtx 1\le i\le2^n\}$ is independent of
$\mathscr F_n$.

So, to prove that $Z_n(s)$ is a martingale, it suffices to prove that,
for $1\le i\le2^n$,
\[
{\mathbf{E}\biggl[ G \bigl(U_i',V_i',h,h,h,h
\bigr) \biggl(\frac{s-\ell_i^n}{r_i^n-\ell
_i^n}\biggr) \Big| \mathscr F_n \biggr]} =h
\biggl( \frac{s-\ell_i^n}{r_i^n-\ell_i^n}\biggr).
\]
Since $U_i',V_i', 1\le i\le2^n$ are independent of $\mathscr F_n$,
this clearly reduces to the following lemma.
%
\begin{lem}\label{lemmartingale}For the operator $G$ defined in \eqref
{eqdefG} and $U,V$ two independent $[0,1]$-uniform random variables
and any $s\in[0,1]$, we have
\[
\mathbf{E} \bigl[G(U,V,h,h,h,h) (s) \bigr] =h(s).
\]
\end{lem}
\begin{pf}Since $V$ and $1-V$ have the same distribution, we have
\begin{eqnarray*}
\mathbf{E} \bigl[G(U,V,h,h,h,h) (s) \bigr]& = &2 \mathbf{E} \biggl[\I{s<U}
(UV)^\beta h\biggl(\frac s U\biggr) \biggr]
\\
&&{}+2 \mathbf{E} \biggl[\I{s\ge U} \bigl((1-U)V \bigr)^\beta h\biggl(
\frac{1-s}{1-U}\biggr) \biggr].
\end{eqnarray*}
Similarly, since $U$ and $1-U$ are both uniform, we clearly have
\[
\mathbf{E} \bigl[G(U,V,h,h,h,h) (s) \bigr] =f(s)+f(1-s),
\]
where we wrote $f(s)=2\E{\I{s<U} (UV)^\beta h(s/U)}$. To complete the
proof, it suffices to compute $f(s)$. We have
\begin{eqnarray*}
f(s)= \mathbf{E} \biggl[\I{s<U} (UV)^\beta h\biggl(\frac s U\biggr)
\biggr] &=&\frac2 {\beta+1} \mathbf{E} \bigl[\I{s<U} s^{\beta/2}(U-s)^{\beta/2}
\bigr]
\\
&=&\frac2 {\beta+1}s^{\beta/2} \int_s^1
(x-s)^{\beta/2}\,dx
\\
&=&\frac4 {(\beta+1) (\beta+2)} s^{\beta/2}(1-s)^{\beta/2+1}
\\
&=&(1-s)h(s),
\end{eqnarray*}
where the last line follows since $(\beta+1)(\beta+2)=4$ by definition
of $\beta$. The result follows readily.
\end{pf}
Our aim is now to prove the following proposition:

\begin{prop}\label{propcontinuouspaths} With probability one $Z_n$
converges uniformly to some continuous limit process $Z$ on $[0,1]$.
\end{prop}

Assume for the moment that there exist constants $a,b\in(0,1)$ and $C$
such that
%
\begin{equation}
\label{eqXborel-cantelli}\mathbf{P}\Bigl(\sup_{s\in[0,1]}\bigl|Z_{n+1}(s)-Z_n(s)\bigr|
\ge a^n\Bigr)
\le C\cdot b^n.
\end{equation}
Then, by the Borel--Cantelli lemma,
the sequences $Z_n$ is almost surely Cauchy with respect to the
supremum norm. Completeness of
$ ( \mathcal C[0,1], \|\cdot\| )$ yields the existence of a
random process $Z$ with continuous paths
such that $Z_n \rightarrow Z$ uniformly on
$[0,1]$.
We now move on to showing that there exist constants $a$ and $b$ such
that~(\ref{eqXborel-cantelli}) is satisfied.
We start by a bound for a fixed value $s\in[0,1]$. We will then handle
the supremum using a sieve of the interval $[0,1]$ by a large enough
number of
deterministic points.

\begin{lem}\label{lemboundfixeds}For every $s\in[0,1]$, any $a\in
(0,1)$, and any integer $n$ large enough, we have the bound
\[
\mathbf{P}\bigl(\bigl|Z_{n+1}(s)-Z_n(s)\bigr|\ge a^n\bigr)
\le4 \bigl(16e
\log(1/a) \bigr)^n.
\]
\end{lem}
\begin{pf}
We use the representation \eqref{eqZn-telescoping}. As we have
already pointed out earlier (Lemma~\ref{lemmartingale}), for every
single rectangle $Q_i^n(s)$ at level $n$, we have
\[
\mathbf{E}\biggl[G \bigl(U_i',V_i',h,h,h,h
\bigr) \biggl(\frac{s-\ell_i^n}{r_i^n-\ell
_i^n}\biggr)-h\biggl(\frac{s-\ell_i^n}{r_i^n-\ell_i^n}\biggr) \Big|
\mathscr F_n \biggr]=0.
\]
Since $h(x)\le2^{-\beta}$ for $x\in(0,1)$, conditional on $\mathscr
F_n$, $Z_{n+1}-Z_n$ is a sum of $2^n$ \emph{centered}, \emph{bounded}
and moreover \emph{independent} terms (but not identically
distributed). Moreover, conditional on $\mathscr F_n$, the term
corresponding to $Q_i^n(s)$ in \eqref{eqZn-telescoping} is bounded by
%
\begin{eqnarray}
\label{boundZn} \vol \bigl(Q_i^n \bigr)^\beta
\cdot\bigl\| G \bigl(U_i',V_i',
h, h, h,h \bigr)-h\bigr\| &\le&\vol \bigl(Q_i^n
\bigr)^\beta2 \|h\|
\nonumber
\\[-8pt]
\\[-8pt]
\nonumber
&=& \vol \bigl(Q_i^n \bigr)^\beta2^{1-\beta}.
\end{eqnarray}
So when conditioning on $\mathscr F_n$, one can bound the variations of
$Z_{n+1}-Z_n$ using the Chernoff--Hoeffding inequality \cite
{Hoeffding1963}. We have
%
\begin{eqnarray}
\label{hoeff}
\mathbf{P}\bigl(\bigl|Z_{n+1}(s)-Z_n(s)\bigr|>a^n\bigr)
&=&
\mathbf{E} \bigl[\mathbf{P} \bigl( {\bigl|Z_{n+1}(s)-Z_n(s)\bigr|>a^n}
| {\mathscr F_n} \bigr) \bigr]
\nonumber
\\
&\le& \mathbf{E} \biggl[2\exp\biggl(-\frac{a^{2n}}{\sum_{i=1}^{2^n} \vol
(Q_i^n(s))^{2\beta}}\biggr) \biggr]
\\
&\le& 2\exp\bigl(-a^{-2n}\bigr)+2
\mathbf{P}\Biggl(\sum_{i=1}^{2^n}
\vol \bigl(Q_i^n(s) \bigr)^{2\beta}>a^{4n}\Biggr);\nonumber
\end{eqnarray}
the precise constant in the exponent in the second inequality can be
taken to be one since it is the case that $2/(2^{1-\beta})^2>1$.

Now, since $2\beta>1$ and all the volumes $\vol(Q_i^n(s))$ are at most
one, we have
%
\begin{eqnarray}
\label{tube}
\mathbf{P}\Biggl(\sum_{i=1}^{2^n} \vol
\bigl(Q_i^n(s) \bigr)^{2\beta}>a^{4n}\Biggr)
 &
\le&
\mathbf{P}\Biggl(\sum_{i=1}^{2^n} \vol
\bigl(Q_i^n(s) \bigr)> a^{4n}\Biggr)
\nonumber
\\[-8pt]
\\[-8pt]
\nonumber
&\le&
\mathbf{P}\bigl(W_n > a^{4n}\bigr),
\end{eqnarray}
where $W_n$ denotes the maximum width of any of the $4^n$ cells at
level $n$. Indeed, the volume occupied by all rectangles $Q_i^n(s)$,
$1\le i\le2^n$ together is at most that of a vertical tube of width
$W_n$. Putting together \eqref{hoeff} and \eqref{tube}, it follows that
\begin{eqnarray*}
\mathbf{P}\bigl(\bigl|Z_{n+1}(s)-Z_n(s)\bigr|\ge a^n\bigr)
 &\le&2\exp
\bigl(-a^{-2n} \bigr) + 2
\mathbf{P}\bigl(W_n> a^{4n}\bigr)
\\
&\le&2\exp \bigl(-a^{-2n} \bigr)+2 \bigl(16e\log(1/a)
\bigr)^n
\\
&\le&4 \bigl(16e\log(1/a) \bigr)^n
\end{eqnarray*}
for all $n$ large enough using Lemma \ref{lemapp1} from the \hyperref[app]{Appendix}.
\end{pf}

Now that we have good control on pointwise variations of $Z_{n+1}-Z_n$,
we move on to the supremum on $[0,1]$. Consider the set $V_n$ of
$x$-coordinates of the vertical boundaries of all the rectangles at
level $n$. Let $L_n=\inf\{|x-y|\dvtx  x,y\in V_n\}$. Suppose that $1/\gamma$
is an integer. Then we have
\begin{eqnarray*}
&&\sup_{s\in[0,1]}\bigl |Z_{n+1}(s)-Z_n(s)\bigr|
\\
&&\qquad\le\sup_{1\le i\le\gamma^{-(n+1)}}\bigl |Z_{n+1} \bigl(i \gamma
^{n+1} \bigr)-Z_n \bigl(i\gamma^{n+1} \bigr)\bigr|\\
&&\qquad\quad{}+2\sup
_{ m \in\{n,n+1\}}\sup_{|s-t|\le
\gamma^{n+1}} \bigl|Z_{m}(s)-Z_{m}(t)\bigr|.
\end{eqnarray*}
We first deal with the second term, and suppose that we are on the event
that $L_{n+1}\ge(4\gamma)^{n+1}$. Observe that the sieve we used,
$\gamma^n$, is much finer than the shortest length of a cell at level
$n+1$ which is at least $L_{n+1}$. We use the representation in \eqref
{Znexplicit}; for $|t-s|\le\gamma^{n+1}$, the two collections $\{
Q_i^n(s), 1\le i\le2^n\}$ and $\{Q_i^n(t), 1\le i\le2^n\}$ differ at
most on one cell. We obtain, for any $|s-t|\le\gamma^{n+1}$,
\begin{eqnarray*}
&&\bigl|Z_n(s)-Z_n(t)\bigr|\\
&&\qquad\le\sum_{i=1}^{2^n}
\vol \bigl(Q_i^n(s) \bigr)^\beta\cdot \biggl
\llvert h\biggl(\frac
{s-\ell_i^n}{r_i^n-\ell_i^n}\biggr)-h\biggl(\frac{t-\ell_i^n}{r_i^n-\ell
_i^n}\biggr) \biggr
\rrvert +2\max_{i}\vol \bigl(Q_i^n(s)
\bigr)^\beta
\\
&&\qquad\le\sum_{i=1}^{2^n} \vol
\bigl(Q_i^n(s) \bigr)^\beta
\cdot4^{-\beta n}+2\max_{i}\vol \bigl(Q_i^n(s)
\bigr)^\beta
\\
&&\qquad\le3 W_n^\beta.
\end{eqnarray*}
Here, the second inequality follows from the facts that $|h(t)-h(s)|\le
|t-s|^\beta$ for any $s,t\in[0,1]$ and that
$ L_n\ge(4\gamma)^{n+1}$. The same upper bound is valid for
$|Z_{n+1}(s) - Z_{n+1}(t)|$ for $|s-t| \leq\gamma^{n+1}$.
In particular, it follows by the union bound that, for any $\gamma\in
(0,1)$ (with $1/\gamma$ an integer),
%
\begin{eqnarray}
\label{eqsupdiffX} &&
\mathbf{P}\Bigl(\sup_{s\in[0,1]} \bigl|Z_{n+1}(s)-Z_n(s)\bigr|
\ge 2 a^n\Bigr)
\nonumber\\
&&\qquad \le \gamma^{-n}\sup_{s\in[0,1]}
\mathbf{P}\bigl(\bigl|Z_{n+1}(s)-Z_n(s)\bigr|\ge a^n\bigr)
\\
&&\qquad\quad{}+
\mathbf{P}\bigl(L_{n+1}<(4\gamma)^{n+1}\bigr)
+
\mathbf{P}\bigl(12 W_n^\beta>a^n\bigr).\nonumber
\end{eqnarray}
We are now ready to complete the proof of Proposition~\ref
{propcontinuouspaths}. From (\ref{eqsupdiffX}) and Lemma~\ref
{lemclosestpair} from the \hyperref[app]{Appendix},
we have
\begin{eqnarray*}
\mathbf{P}\Bigl(\sup_{s\in[0,1]} \bigl|Z_{n+1}(s)-Z_n(s)\bigr|\ge2
a^n\Bigr)
 &\le& 4 \bigl(16 e \gamma^{-1} \log(1/a)
\bigr)^n+ 6 \cdot16^{n} \gamma^{n/201}
\\
&&{}+ \bigl(4e\log \bigl(12^{1/n} /a \bigr)/\beta \bigr)^n
\end{eqnarray*}
for all $\gamma< \gamma_0 /4$ and $n \geq n_0(\gamma,a)$. Now, first choose
$a <1$ sufficiently close to $1$ such that we also have $16 (e \log
(1/a))^{1/202} < 1/4$ and then $\gamma> 0$ such that
$1/\gamma$ is an integer and $\gamma^{1/201} \leq e\gamma^{-1} \log(1/a)$.

It follows that, for $n$ sufficiently large,
\[
\mathbf{P}\Bigl(\sup_{s\in[0,1]}\bigl |Z_{n+1}(s)-Z_n(s)\bigr|\ge2
a^n\Bigr)
 \leq11 \cdot4^{-n}.
\]
Increasing $a<1$ and $C$ ensures that \eqref{eqXborel-cantelli} holds
with $b=1/4$ for all $n\ge1$.
The functions $Z_n^{1}, \ldots, Z_n^{4}$ at the four children of the
root are each distributed as $Z_{n-1}$, and they also converge
uniformly to continuous limits denoted $Z^{(1)}, \ldots, Z^{(4)}$. The
random functions $Z^{(1)}, \ldots, Z^{(4)}$ are independent and
distributed as $Z$. Equation~(\ref{defrecZ}) and independence imply
\begin{eqnarray*}
Z(s)& = &\I{s<U} \biggl[(UV)^\beta Z^{(1)}\biggl(\frac s U
\biggr)+ \bigl(U(1-V) \bigr)^\beta Z^{(2)}\biggl(\frac s U
\biggr)\biggr]
\\
&&{}+\I{s\ge U} \biggl[ \bigl((1-U)V \bigr)^\beta Z^{(3)}
\biggl( \frac{s-U}{1-U}\biggr)\\
&&\qquad\qquad{}+ \bigl((1-U) (1-V) \bigr)^\beta
Z^{(4)}\biggl( \frac{s-U}{1-U}\biggr)\biggr],
\end{eqnarray*}
almost surely, considered as random continuous paths. In particular,
the distribution of $Z$ solves the distributional fixed-point equation
(\ref{eqlimitprocess}).

Finally, we look at the moments of $\| Z_n\|=\sup_{s\in[0,1]}|Z_n(s)|$
and $\|Z\|=\sup_{s\in[0,1]}|Z(s)|$.
%
\begin{prop}\label{promomentslimit}For every $p\ge1$, we have $\Ec{\|
Z\|^p}<\infty$ and $\|Z_n-Z\|\to0$ in $L^p$.
\end{prop}
\begin{pf}
Let $\Delta(x) = \Prob{\| Z_{n+1} - Z_n \| \geq x}$
and $a < 1, C > 0$ such that \eqref{eqXborel-cantelli} is satisfied
with $b = 1/4$.
Then, by \eqref{eqZn-telescoping} and the upper bound $\eqref
{boundZn}$, we have
%
\begin{equation}
\label{dingens} \qquad\mathbf{E} \bigl[\|Z_{n+1} - Z_n \| \bigr] = \int
_0^{\infty} \Delta_n(x) \,dx = \int
_0^{a^n} \Delta_n(x) \,dx + \int
_{a^n}^{2^{n+1}} \Delta_n(x) \,dx.
\end{equation}
The first summand is at most $a^n$, the second one at most $C \cdot
2^{-(n-1)}$ by \eqref{eqXborel-cantelli}.
Altogether, there exist $R > 0$ and $0 < q < 1$ with
\[
\mathbf{E}\bigl [\|Z_{n+1} - Z_n \| \bigr] \leq R q^n
\]
for all $n$.
Furthermore, for any $p \in\N$, our proof also provides \eqref
{eqXborel-cantelli} for a constant $C > 0$ and $b = 4^{-p}$ by
increasing the value of $a$.
Therefore, replacing $a^n$ and $2^{n+1}$ by $a^{np}$, respectively, $2^{(n+1)p}$
in \eqref{dingens} shows that the $p$th moment of $\|Z_{n+1} - Z_n \|$
is also exponentially small in $n$ for any $p > 1$.
Then, since $Z_n=h+\sum_{k=1}^n (Z_k-Z_{k-1})$, using Minkowski's inequality,
\[
\mathbf{E} \bigl[\|Z_n\|^p \bigr] ^{1/p} \leq
\sum_{k=1}^n \mathbf{E} \bigl[\llVert
Z_{k} - Z_{k-1} \rrVert ^p
\bigr]^{1/p} + \|h\|,
\]
which is uniformly bounded in $n$. It follows that $\Ec{\| Z\|^p} <
\infty$ for all $p\ge1$, and that
$\Ec{\|Z_n - Z \|^p}\to0$ as $n\to\infty$.
\end{pf}
%

\section{Uniform convergence of the mean}\label{secunifconvergence}

The proof that assumption (A1) holds for Proposition~\ref{propcont}
requires that we show convergence of the first moment $n^{-\beta}\E
{C_n(s)}$ toward $\mu_1(s)=K_1 h(s)$ uniformly on $[0,1]$. Note that,
since $C_n(s)$ is continuous at any fixed $s\in[0,1]$ almost surely,
the function $s \to\Ec{C_n(s)}$ is continuous for any $n$.
Curien and Joseph \cite{CuJo2010} only show point-wise convergence, and proving uniform
convergence requires a good deal of additional arguments.
Unfortunately, a good portion of the work consists of a tedious
tightening of the strategy developed in~\cite{CuJo2010}.
%
\begin{prop}\label{PROPUNIFORMCONVERGENCE}There exists $\varepsilon
>0$ such that
\[
\sup_{s \in[0,1]} \bigl|n^{-\beta} \mathbf{E} \bigl[C_n(s)
\bigr] - \mu_1(s)\bigr | = O \bigl(n^{-\varepsilon} \bigr).
\]
In other words, $n^{-\beta}\Ec{C_n(s)}$ converges uniformly to $\mu_1$
on $[0,1]$ with polynomial rate.
\end{prop}

We prove a Poissonized version. Since $C_n(s)$ is increasing in $n$ for
every fixed~$s$, the de-Poissonization only relies on routine arguments
based on concentration for Poisson random variables, and we omit the
details. Consider a Poisson point process with unit intensity on
$[0,1]^2\times[0,\infty)$. The first two coordinates represent the
location inside the unit square; the third one represents the time of
arrival of the point. Let $P_t(s)$ denote the partial match cost for a
query at $x=s$ in the quadtree built from the points arrived by time $t$.

\begin{prop}\label{propuniformpoisson}There exists $\varepsilon>0$
such that
\[
\sup_{s \in[0,1]} \bigl|t^{-\beta} \mathbf{E} \bigl[P_t(s)
\bigr] -\mu_1(s)\bigr| = O \bigl(t^{-\varepsilon} \bigr).
\]
\end{prop}

The proof of Proposition~\ref{propuniformpoisson} relies crucially on
two main ingredients: first, a~strengthening of the arguments developed by
Curien and Joseph \cite{CuJo2010}, and the speed of convergence $\E{C_n(\xi)}$ to $\Ec
{\mu_1(\xi)}$ for a uniform query line $\xi$; see \eqref{hwangquad}.
By symmetry, we write for any $\delta\in(0,1/2)$,
%
\begin{eqnarray}
\label{equniformdecomp} &&\sup_{s \in[0,1]}\bigl |t^{-\beta}
\mathbf{E} \bigl[P_t(s)\bigr] - \mu_1(s)\bigr| \nonumber\\
&&\qquad= \sup
_{s\in[0,1/2]}\bigl |t^{-\beta} \mathbf{E} \bigl[P_t(s)
\bigr] -\mu_1(s) \bigr|
\\
&&\qquad\le\sup_{s\le\delta} \bigl|t^{-\beta} \mathbf{E}
\bigl[P_t(s)\bigr]- \mu_1(s) \bigr|+\sup_{s\in(\delta,1/2]}
\bigl|t^{-\beta} \mathbf{E} \bigl[P_t(s)\bigr] -
\mu_1(s) \bigr|.\nonumber
\end{eqnarray}
The two terms on the right-hand side above are controlled by the
following lemmas.
%
\begin{lem}[(Behavior on the edge)]\label{lemuniformedge} We have
%
\begin{equation}
\label{equniformedge} \sup_{s\le\delta}\bigl |t^{-\beta} \mathbf{E}
\bigl[P_t(s)\bigr] -\mu_1(s) \bigr| \le2^\beta
\sup_{r \geq t/2} r^{-\beta} \mathbf{E} \bigl[P_r(
\delta) \bigr] + K_1\delta^{\beta/2}.
\end{equation}
\end{lem}

\begin{lem}[(Behavior away from the edge)]\label{lemuniformmiddle}There
exist constants $C_1,C_2,\eta$ with $0 < \eta< \beta$ and $\gamma\in
(0,1)$ such that, for any integer $k$ and real number $\delta\in
(0,1/2)$ we have, for any real number $t>0$,
\[
\sup_{s\in[\delta,1/2]} \bigl|t^{-\beta} \mathbf{E} \bigl[P_t(s)
\bigr]- \mu_1(s)\bigr| \le C_1 \delta ^{-1} (1-
\gamma)^k + C_2 k 2^k (\beta-
\eta)^{-2k} t^{-\eta}.
\]
\end{lem}

Before going further, we indicate how these two lemmas imply
Proposition~\ref{propuniformpoisson}. By Lemmas~\ref
{lemuniformedge} and
\ref{lemuniformmiddle}, we have for any $\delta\in(0,1/2)$ and
natural number $k\ge0$
\begin{eqnarray*}
&&\sup_{s \in[0,1]} \bigl|t^{-\beta} \mathbf{E} \bigl[P_t(s)
\bigr] - \mu_1(s)\bigr| \\
&&\qquad\le 3 K_1 \delta^{\beta/2}+ 3
C_1 \delta^{-1} (1-\gamma)^k + 5
C_2 k t^{-\eta}2^k (\beta-\eta)^{-2k}.
\end{eqnarray*}
Choosing $\delta=t^{-\nu}$ and $k=\lfloor\alpha\log t\rfloor$ for $\nu,\alpha>0$ to be determined, we obtain
\begin{eqnarray*}
\sup_{s \in[0,1]} \bigl|t^{-\beta} \mathbf{E} \bigl[P_t(s)
\bigr] - \mu_1(s)\bigr| &\le & 3 K_1 t^{-\nu\beta/2} + 3
C_1 t^{\nu} (1-\gamma)^{\alpha\log t-1}
\\
& &{}+ 5 C_2 t^{-\eta} \bigl[2/(\beta-\eta)^2
\bigr]^{\alpha\log t} \alpha\log t.
\end{eqnarray*}
First pick $\alpha>0$ small enough that
\[
\alpha\log\biggl(\frac2{(\beta-\eta)^2}\biggr)<\eta.
\]
This $\alpha$ being fixed, choose $\nu>0$ small enough that $\nu+\alpha
\log(1-\gamma)<0$. The claim follows.

Since Curien and Joseph \cite{CuJo2010} prove convergence at any $s\in(0,1)$, it comes
as no surprise that the convergence may be strengthened to uniform
convergence on compacts of $(0,1)$ by checking carefully the (long)
sequence of bounds in \cite{CuJo2010} (Lemma~\ref{lemuniformmiddle}).
We provide the details in the \hyperref[app]{Appendix} for the sake of completeness.
The behavior at the edge, however (Lemma~\ref{lemuniformedge}),
consists precisely of controlling what happens when the bounds in \cite
{CuJo2010} do not work any longer; this is why we provide here the
additional arguments.
To deal with the term involving the values of $s\in[0,\delta]$, we
relate the value $\Ec{P_t(s)}$ to $\Ec{P_t(\delta)}$. The term $\Ec
{P_t(\delta)}$ will then be shown to be small using the pointwise
convergence and choosing $\delta$ small.

The function $\mu_1(s)=\lim_{t\to\infty}\Ec{P_t(s)}$ is monotonic for
$s\in[0,1/2]$. It seems, at least intuitively, that for any fixed real
number $t>0$, $\Ec{P_t(s)}$ should also be monotonic for $s\in[0,1/2]$,
but we were unable to prove it. The following weaker version will be
sufficient for our needs.

\begin{prop}[(Almost monotonicity)]\label{propmonotonicity}
For any $s<1/2$ and $\varepsilon\in[0,\break1-2s)$, we have
\[
\mathbf{E} \bigl[P_t(s)\bigr] \le \mathbf{E} \biggl[P_{t(1+\varepsilon)}
\biggl(\frac{s+\varepsilon
}{1+\varepsilon}\biggr) \biggr].
\]
\end{prop}

The idea underlying Proposition~\ref{propmonotonicity} requires that
we understand what happens to the quadtree upon considering a larger
point set.
For a finite point set $\mathcal P\subset[a,b] \times[0,1] \times
[0,\infty)$, we
let $V(\mathcal P)$ and $H(\mathcal P)$ denote, respectively, the set
of vertical and horizontal line segments of the quadtree built from
$\mathcal P$.

\begin{lem}\label{leminsertion}Let $\mathcal P = \{p_1, \ldots, p_n\}$
be a set of points with $p_i = (x_i, y_i, t_i) \in[a_2,a_3]\times
[0,1] \times[0,\infty)$ ordered by their $t$ coordinate, that is, $t_i
\leq t_{i+1}$. Additionally we assume $\mathcal P$ to be in general
position, meaning that all $x$-coordinates are pairwise different, and
the same holds true for the $y$ and $t$ coordinates. Furthermore let
$\mathcal Q = \{p_1', \ldots, p_m'\} \subseteq[a_1,a_2]\times[0,1]
\times[0,\infty)$ with $p_i' = (x_i', y_i', t_i')$ again ordered
according to their third coordinate such that $\mathcal P \cup\mathcal
Q \subseteq[a_1, a_3] \times[0,1] \times[0, \infty)$ is again in
general position. Then we have
\[
H(\mathcal P\cup\mathcal Q) \supset H(\mathcal P) \quad\mbox {and}\quad V(\mathcal P \cup
\mathcal Q)\subset V(\mathcal P).
\]
\end{lem}
\begin{pf}
We assume for a contradiction that the assertion is wrong and focus on
the case that $H(\mathcal P) \not\subset H(\mathcal P\cup\mathcal Q)$;
the other case is
handled analogously.
Let $i_1$ be the index of the ``first'' point in $\mathcal P$ such that
the horizontal line of $p_{i_1}$ is shorter (at least on the right or left-hand
side of the point)
in the quadtree built from $\mathcal P \cup\mathcal Q$ than it is in
the one built from $\mathcal P$. Here, first refers to the time
coordinate $t$.
Now, by construction there must be an index $i_2$ such that the
vertical line of $p_{i_2}$ blocks the horizontal line of $p_{i_1}$ in
$\mathcal P \cup\mathcal Q$ but not in $\mathcal P$. We again choose
$i_2$ such that $t_{i_2}$ is minimal with this property; by
construction $t_{i_2} < t_{i_1}$. Repeating the argument gives the
existence of an index $i_3$ and a point $p_{i_3}$ whose horizontal line
blocks the vertical line of $p_{i_2}$ in $\mathcal P$ but not in
$\mathcal P \cup\mathcal Q$ with $t_{i_3} < t_ {i_2}$. This obviously
contradicts the choice of $i_1$.
\end{pf}

\begin{pf*}{Proof of Proposition~\ref{propmonotonicity}}Consider the
unit square $[0,1]^2$ and the extended box $[-\varepsilon, 1]\times
[0,1]$, and a single Poisson point process on $[-\varepsilon, 1]\times
[0,1]\times[0,t]$ with unit intensity. Write $P_t^\varepsilon(s)$ for
the number of (horizontal) lines intersecting $\{x=s\}$ in the quadtree
formed by all the points. Similarly, let $P_t(s)=P_t^0(s)$ be the
corresponding quantity when the quadtree is formed using only the
points falling inside $[0,1]^2$. Then, for this coupling, we have by
Lemma~\ref{leminsertion},
\[
P_t(s)\le P_t^\varepsilon(s)\eqdist
P_{t(1+\varepsilon)}\biggl(\frac
{s+\varepsilon}{1+\varepsilon}\biggr).
\]
Taking expectations completes the proof.
\end{pf*}

\begin{pf*}{Proof of Lemma~\ref{lemuniformedge}}
We use Proposition~\ref{propmonotonicity} to relate $\Ec{P_t(s)}$ to
$\Ec{P_{t'}(\delta)}$ for some $t'$. Choosing $\varepsilon=(\delta
-s)/(1-\delta)$ yields $t'=t (1-s)/(1-\delta)\le t(1-\delta)^{-1}$.
Thus, for any $\delta\in(0,1/2)$ and $t>0$ we have
\begin{eqnarray*}
&&\sup_{s\le\delta} \bigl|t^{-\beta} \mathbf{E} \bigl[P_t(s)
\bigr] - \mu_1(s) \bigr|\\
 &&\qquad \le\sup_{s\le\delta}
t^{-\beta} \mathbf{E} \bigl[P_t(s)\bigr] + \mu_1(
\delta)
\\
& &\qquad\le\sup_{s\le\delta} t^{-\beta} \mathbf{E}
\bigl[P_{t'}( \delta)\bigr] +\mu_1(\delta )
\\
&&\qquad \le t^{-\beta} \mathbf{E} \bigl[P_{t/(1-\delta)}(\delta)\bigr] +
\mu_1( \delta)
\\
&&\qquad \le(1-\delta)^{-\beta} \sup_{r\ge t/2 } r^{-\beta}
\mathbf{E} \bigl[P_r(\delta)\bigr] + \mu_1(\delta).
\end{eqnarray*}
This completes the proof since $\delta\leq\frac{1}{2}$ and $\mu
_1(s)\le K_1 \delta^{\beta/2}$.
\end{pf*}

\section{\texorpdfstring{Moments and supremum: Proofs of Theorems~\protect\ref{thmsupremum},
\protect\ref{thmsame-Xi} and Corollary~\protect\ref{thmmixedn}}
{Moments and supremum: Proofs of Theorems 4, 5 and Corollary 6}}\label{secmomentssup}

Our main result implies the convergence of the second moment of the
discrete toward that of the limit process. This section is devoted to
identifying this limit; in particular, it provides an explicit
expression for the limit variance.

We first focus on the moments. The definition of the process $Z(s)$
implies that the second moment $\mu_2(s)=\Ec{Z(s)^2}$ satisfies an
integral equation. We have 
%
\begin{eqnarray*}
\mu_2(s) 
&=& 2 \mathbf{E} \bigl[Y^{2\beta} \bigr]
\biggl\{\int_s^1 x^{2\beta}\cdot
\mu_2 \biggl(\frac s x\biggr)\,dx + \int_0^s
(1-x)^{2\beta}\cdot\mu_2\biggl(\frac{1-s}{1-x}\biggr)\,dx
\biggr\}
\\
&&\hspace*{-2pt}{}+ 2 \mathbf{E} \bigl[ \bigl[Y(1-Y) \bigr]^\beta \bigr] \cdot \biggl\{
\int_s^1 \!\!x^{2\beta}h\biggl(\frac s x
\biggr)^2\,dx + \int_0^s\!
(1-x)^{2\beta}h\biggl(\frac{1-s}{1-x}\biggr)^2\,dx \biggr\}.
\end{eqnarray*}
It now follows that $\mu_2$ satisfies the following integral equation:
\begin{eqnarray*}
\mu_2(s)&=&\frac2{2\beta+1} \biggl\{\int_s^1
x^{2\beta} \mu_2\biggl(\frac s x\biggr)\,dx + \int
_0^s (1-x)^{2\beta} \mu_2
\biggl(\frac{1-s}{1-x}\biggr)\,dx \biggr\}
\\
&&{}+ 2 \mathrm{B} (\beta+1, \beta+1 )\cdot\frac{h^2(s)}{\beta+1}.
\end{eqnarray*}
One easily verifies that the function $f$ given by $f(s)=c_2 h^2(s)$
solves the above equation when $c_2$ is given by
%
\begin{equation}
\label{eqKbar} c_2= 2 \mathrm{B} (\beta+1, \beta+1 )
\frac{2\beta
+1}{3(1-\beta)}.
\end{equation}

In order to show that $\mu_2=c_2h(s)^2$, it now suffices to prove that
the integral equation satisfied by $\mu_2$ admits a unique solution in
a suitable function space.
To this end, we show that the map $K$ defined below is a contraction
for the supremum norm:
%
\begin{eqnarray}
\label{lipop} Kf(s)&=&\frac{2}{2\beta+1} \biggl\{\int_s^1
x^{2\beta} f\biggl(\frac sx\biggr)\,dx + \int_0^s
(1-x)^{2\beta} f\biggl(\frac{1-s}{1-x}\biggr)\,dx \biggr\}
\nonumber
\\[-8pt]
\\[-8pt]
\nonumber
&&{}+2 \mathrm{B} (\beta+1, \beta+1 ) \frac{h(s)^2}{\beta+1}.
\end{eqnarray}
For any two functions $f$ and $g$, measurable and bounded on $[0,1]$,
we have
\begin{eqnarray*}
&&\|K f- Kg\|
\\
&&\qquad = \frac{2}{2\beta+1}\sup_{s\in[0,1]} \biggl\llvert \int
_s^1 x^{2\beta}\biggl(f \biggl(\frac sx
\biggr)-g \biggl(\frac sx\biggr)\biggr)\,dx\\
&&\hspace*{75pt}\qquad{}+\int_0^s
(1-x)^{2\beta} \biggl( f \biggl(\frac{1-s}{1-x}\biggr)-g \biggl(
\frac{1-s}{1-x}\biggr) \biggr)\,dx \biggr\rrvert
\\
&&\qquad \le\frac{2}{2\beta+1} \biggl(\sup_{s\in[0,1]} \biggl\{\int
_s^1 x^{2\beta}\,dx \biggr\}+ \sup
_{s\in[0,1]} \biggl\{\int_0^s
(1-x)^{2\beta
}\,dx \biggr\} \biggr) \|f-g\|
\\
& &\qquad= \frac{4}{(2\beta+1)^2} \|f-g\|.
\end{eqnarray*}
Since $2\beta+1>2$, the operator $K$ is a contraction on the set of
measurable and bounded functions on $[0,1]$ equipped with the supremum norm.
Banach fixed point theorem then ensures that the fixed point is unique,
which shows that indeed $\Ec{Z(s)^2} = c_2 h^2(s).$ Then, $K_2 = c_2
-1$ and one obtains easily the expression for $\Vc{Z(\xi)}$ in \eqref
{eqvarzxi} by integration.

Analogously one shows that
the $m$th moment of $Z(s)$ is of the form $c_m h(s)^m$ where $c_m$
solves \eqref{recmomentsZ}.
The Lipschitz constant of the corresponding operator in~\eqref{lipop}
is $4 / (\beta m + 1)^2$, hence again smaller than one. This
immediately implies
that $(c_m)_{m \geq1}$ are the moments of $Z(s) / h(s),$ independently
of $s$.

Furthermore, there is only one distribution with these moments. We let
$\Psi$ denote the corresponding random variable. To prove this, we show
that there exists a constant $A_1 > 0$ such that
%
\begin{equation}
\label{boundmom} c_m \leq A_1^m
m^m,\qquad m \geq1,
\end{equation}
which completes the proof of the proposition by the Carleman condition;
see, for example, \cite{Feller1971}, page 228.

Suppose that \eqref{boundmom} is satisfied for all $m < m_0$. By
Stirling's formula, there exists a constant $A_2$ such that for all
$m \geq1$ and $1\leq\ell< m,$
\[
\pmatrix{m \cr\ell} \mathrm{B} \bigl(\beta\ell+1, \beta(m-\ell)+1 \bigr) \leq
\frac{A_2}{m} \biggl( \frac{\ell^\ell(m-\ell)^{m-\ell}}{m^m} \biggr)^{\beta-1}.
\]
Next, the prefactor in \eqref{recmomentsZ} is of order $1/m$, and
hence bounded by
$A_3 /m$ for some $A_3> 0$ and all $m > 1$.
Using this, the induction hypothesis and $x^x(1-x)^{1-x} \leq1$ for
all $x \in[0,1]$ it follows that
\begin{eqnarray*}
c_{m_0}& \leq &\frac{A_2 A_3} {m_0^2}\sum_{\ell=1}^{m_0-1}
\bigl( \ell ^\ell(m_0-\ell)^{m_0-\ell}
\bigr)^{\beta-1} m_0^{m_0(1-\beta)} c_\ell
c_{m_0-\ell}
\\
& \leq& \frac{A_1^{m_0} A_2 A_3} {m_0^2} \sum_{\ell=1}^{m_0-1}
m_0^{\beta m_0} m_0^{m_0(1-\beta)}
\\
& \leq& A_1^{m_0} m_0^{m_0},
\end{eqnarray*}
if $m_0$ is chosen large enough. Finally, it is easy to see that any
solution of \eqref{fixedone} with unit mean and finite second moment
has finite moments of all orders. Thus, its moments also satisfy \eqref
{recmomentsZ} and it must coincide with $\Psi$ in distribution.

We now consider the supremum $S_n=\sup_{s\in(0,1)} C_n(s)$. The
uniform convergence of $n^{-\beta}C_n$ directly implies, as $n\to\infty$,
\[
\bar S_n:= \frac{S_n}{K_1 n^{\beta}} \rightarrow S
\]
in distribution with $S = \sup_{t \in[0,1]} Z(t)$
where $Z$ is the process constructed in Section~\ref{seclimit}. The
results obtained so far yield that,
stochastically,
\begin{eqnarray*}
S &\leq& \bigl( (UV)^{\beta} S^{(1)} + \bigl(U(1-V)
\bigr)^{\beta} S^{(2)} \bigr)\\
&&{} \vee \bigl( \bigl((1-U)V
\bigr)^{\beta} S^{(3)} + \bigl((1-U) (1-V) \bigr)^{\beta}
S^{(4)} \bigr),
\end{eqnarray*}
where $S^{(1)}, \ldots, S^{(4)}$ are independent copies of $S$, also
independent of $(U,V)$ which are themselves independent and uniform on $[0,1]$.
To complete the proof of Theorem~\ref{thmsupremum}, it remains to
prove that, for all $m$,
$\E{S^m} < \infty$ and that $\Ec{\bar S_n^m} \to\Ec{S^m}$, as $n \to
\infty$. Theorem 12 and Corollary 21 in \cite{NeSu2011a} provide
uniform integrability of $\bar S_n^2$.
It follows that $\bar S_n$ is bounded in $L^2$ and hence also in $L^1$.
For higher moments, we proceed by induction. Let $B_1$ be such that $\Ec
{\bar S ^m_n} \leq B_1$ for all $m < m_0$ and $n \geq1$ with $m_0 \geq
2$. Furthermore, choose $B_2$ such that $\Ec{\bar S^{m_0}_n} \leq B_2$
for all $n < n_0$.
Then, the recurrence for $C_n(t)$ yields
\begin{eqnarray*}
\mathbf{E} \bigl[\bar S^{m_0}_{n_0} \bigr] &\leq& \mathbf{E}
\biggl[ \biggl( \biggl( \frac{I_1^{(n)}}{n} \biggr)^{\beta} \bar
S_{I_1^{(n)}}^{(1)} + \biggl(\frac{I_2^{(n)}}{n}
\biggr)^{\beta} \bar S_{I_2^{(n)}}^{(1)} \biggr)^{m_0}
\biggr]
\\
&&{}+ \mathbf{E} \biggl[ \biggl( \biggl(\frac{I_3^{(n)}}{n} \biggr)^{\beta}
\bar S_{I_3^{(n)}}^{(3)} + \biggl(\frac{I_4^{(n)}}{n}
\biggr)^{\beta} \bar S_{I_4^{(n)}}^{(4)} \biggr)^{m_0}
\biggr]
\\
&\leq& 4^{m_0} B_1^2 + 4 B_2
\mathbf{E} \biggl[ \biggl(\frac{I_1^{(n)}}{n} \biggr)^{\beta m_0} \biggr].
\end{eqnarray*}
Note that, as $n \to\infty$, we have
$\Ec{(I_1^{(n)}/n)^{\beta m_0}} \rightarrow\E{(UV)^{\beta m_0}} =
(\beta m_0 +1)^{-2},$
thus choosing $n_0$ and $B_2$ appropriately we have $\Ec{ \bar
S^{m_0}_{n_0}} \leq B_2$ since $m_0 \geq2$. This shows that $\bar S_n$
is bounded in $L^{m_0}$, and the assertion follows.

\section{Partial match queries in random $2$-d trees}\label{seckd}
\subsection{$2$-d trees: Constructions and recursions}
The random $2$-d tree was introduced by Bentley \cite{Bentley1975} and is used
to store two-dimensional data just as the two-dimensional quadtree. It
is also called two-dimensional binary search tree since it is binary
and mimics the construction rule of binary search tree for
two-dimensional data. Our aim in this section is to introduce $2$-d
trees, and extend to $2$-d trees the results for partial match queries
in quadtrees we obtained in the previous sections. All the results can
be transferred (convergence as a process, convergence of all moments at
one or multiple points, convergence of the supremum in distribution and
for all moments); we will mainly state the forms of the theorems for
$2$-d trees, and focus on the points that deserve some verifications.

\textsc{Construction of $2$-d trees.} The data are partitioned recursively, as in quadtrees, but the splits
are only binary; since the data is two-dimensional, one alternates
between vertical and horizontal splits, depending on the parity of the
level in the tree. More precisely, consider a point sequence
$p_1,p_2,\ldots, p_n\in[0,1]^2$. As we build the tree, regions are
associated to each node. Initially, the root is associated with the
entire square $[0,1]^2$. The first item $p_1$ is
stored at the root, and splits \emph{vertically} the unit square in two
rectangles, which are associated with the two children of the root.
More generally, when $i$ points have already been inserted, the tree
has $i$ internal nodes, and $i+1$ (lower level) regions associated to
the external nodes and forming a partition of the square $[0,1]^2$.
When point $p_{i+1}$ is stored in the node, say $u$, corresponding to
the region it falls in, divides the region in two sub-rectangles that
are associated to the two children of $u$, which become external nodes;
that last partition step depends on the parity of the depth of $u$ in
the tree: if it is odd we partition horizontally, if it is even we
partition vertically. See Figure~\ref{figkd}. (Of course, one could
start at the root with a horizontal split.)

\begin{figure}

\includegraphics{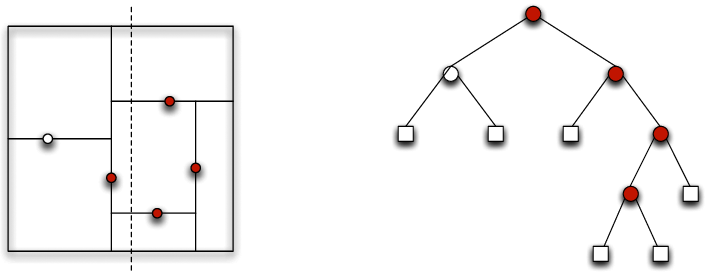}

\caption{An example of a $2$-d tree is shown: on the
left, the partition of $[0,1]^2$ induced by the points; on the right,
the corresponding binary tree. The colored nodes are the ones that are
visited when performing the partial match query materialized by the dashed line.}\label{figkd}
\end{figure}

\textsc{Partial match queries.} From now on, we assume that data consists of a set of independent
random points, uniformly distributed on the unit square. Unlike in the
case of quadtrees, the \emph{direction} of a partial match query line
with respect to the direction of the root does matter. Let $C_n^=(t)$
and $C^\perp_n(t)$ denote the number of nodes visited by a partial
match for a query at position $t\in[0,1]$ when the directions of the
split at the root and the query are parallel and perpendicular,
respectively. Subsequently, we will
analyze both quantities synchronously as far as possible. We will
always consider directions with respect to the query line, and although
some of the expressions (for the sizes of the regions, e.g.)
will be symmetric, we keep them distinct for the sake of clarity. (We
also assume without loss of generality that the query line is always
vertical, and that the direction of the cut at the root may change.)

As in a quadtree, a node is visited by a partial match query if and
only if it is inserted in a subregion that intersects the query line.
Unfortunately, these nodes are not easily
identifiable \emph{after} the insertion of $n$ points; the value of the
quantity $C_n^=(s)$ is obtained by adding twice the number of lines
intersecting the query line
at $s$ and the number of boxes that are intersected by the query line
and will have their next split perpendicular to the query line (i.e.,
the depth of the
corresponding external nodes in the tree have odd parity).

\textsc{Recursive decompositions.} Let $(U,V)$ be the first point which partitions the unit square. By
construction, since the directions of the partitioning lines alternate,
both processes $C_n^=(t)$ and $C_n^\perp(t)$ are coupled: when the
query line is perpendicular to the split direction, the recursive
search occurs in both child sub-regions whose sizes we denote by $N_n$
and $S_n$, and we have
%
\begin{equation}
\label{recHV} C_n^\perp(s) \eqdist 1 +
C^{(=, 1)}_{N_n} (s) + C^{(=, 2)}_{S_n}(s);
\end{equation}
when the query line and the first split at the root are parallel, only
one of the sub-regions (of sizes $L_n$ and $R_n$) is recursively
visited, and we have
%
\begin{equation}
\label{recVH} C_n^\perp(s) \eqdist1+ \I{s<U}
C^{(=, 1)}_{L_n} \biggl(\frac s U\biggr) + \I {s \geq U}
C^{(=, 2)}_{R_n} \biggl(\frac{s-U} {1-U}\biggr).
\end{equation}
Here $(C^{(=, 1)}_n)_{n\geq0}, (C_n^{(=, 2)})_{n\geq0}$ are
independent copies of $(C^=_n)_{n\geq0}$, independent of $(N_n, S_n)$
in \eqref{recHV} and
$(C^{(\perp, 1)}_n)_{n\geq0}$, $(C^{(\perp, 2)})_{n\geq0}$ are
independent copies of $(C^\perp_n)_{n\geq0}$, independent of $(L_n,
R_n)$ in \eqref{recVH}.
Moreover, here and in the following distributional recurrences and
fixed-point equations involving a parameter $s \in[0,1]$ are to be
understood on the level of c{\`a}dl{\`a}g or continuous functions
unless stated otherwise.

As in the case of partial match in random quadtrees, the expected value
at a random uniform query line $\xi$, independent of the tree is of
order $n^{\beta}$ for the same constant $\beta$ defined in \eqref
{defbeta}, and we have
\[
\mathbf{E} \bigl[C_n^=(\xi)\bigr] \sim\kappa_= n^{\beta},\qquad
\mathbf{E} \bigl[C_n^\perp( \xi)\bigr] \sim
\kappa_\perp n^{\beta}
\]
for some constants $\kappa_= > 0, \kappa_\perp> 0$.
This was first proved by Flajolet and Puech~\cite{FlPu1986}. A more detailed analysis by
Chern and Hwang \cite{ChHw2006} shows that
%
\begin{eqnarray}
\qquad\mathbf{E} \bigl[C_n^=(\xi)\bigr] &= &\kappa_= n^{\beta} - 2 +
O \bigl(n^{\beta-1} \bigr), \qquad\kappa_= = \frac{13(3-5 \beta)}{4}\cdot
\frac{\Gamma(2\beta
+2)}{\Gamma(\beta+1)^3}, 
\label{eqk1}
\\
\mathbf{E} \bigl[C_n^\perp(\xi)\bigr] &=&
\kappa_\perp n^{\beta} - 3 + O \bigl(n^{\beta-1} \bigr),\qquad
\kappa_\perp= \frac{13(2 \beta-1)}{2} \cdot\frac{\Gamma
(2\beta+2)}{\Gamma(\beta+1)^3} 
\label{eqk2}.
\end{eqnarray}
Observe that $\kappa_= = \frac12 13(3-5 \beta) \kappa$ and $\kappa
_\perp= 13(2\beta-1) \kappa$, where $\kappa$ is the leading constant
for $\Ec{C_n(\xi)}$ in the case of quadtrees defined in \eqref{defbeta}.

\textsc{Homogeneous recursive relations and limit behavior.} For our purposes, and although it yields more complex expressions, it
is more convenient to expand the recursion one more level to obtain
recursive relations that only involve quantities of the same type, only
$(C_n^=)_{n\ge0}$ or only $(C_n^\perp)_{n\ge0}$: each one of the
first two sub-region at the root is eventually split, and this gives
rise to a partition into four regions at level two of the tree. Let
$(U_\ell,V_\ell)$ and $(U_r,V_r)$ be, respectively, the first points on
each side (left and right) of the first cut, when it is parallel to the
query line. Let also $(U_u,V_u)$ and $(U_d,V_d)$ be the first points on
each side of the cut (up and down) when it is perpendicular to the
query line. Note that $U,V_\ell,V_r$ are independent and uniform on
$[0,1]$, and so are $V,U_u$ and $U_d$.

Let $I_{=,1}^{(n)}, \ldots, I_{=,4}^{(n)}$ and $I_{\perp,1}^{(n)},\ldots, I_{\perp,4}^{(n)}$ denote the number of data points falling in these
regions when the root and the query line are parallel and
perpendicular,
respectively. The distributions of $I_{=,1}^{(n)},\ldots, I_{=,4}^{(n)}$
on the one hand, and $I_{\perp,1}^{(n)},\ldots, I_{\perp,4}^{(n)}$ on
the other hand are slightly more involved than in the case of
quadtrees. One has,
for example, given the values of $U, V_\ell,V_r$ it holds
\[
I_{=,1}^{(n)} \eqdist\operatorname{Bin} \bigl( \bigl(
\operatorname{Bin}(n-1; U)-1 \bigr)_+, V_\ell \bigr)
\]
and given $V, U_d,U_u$
\[
I_{\perp,1}^{(n)} \eqdist\operatorname{Bin} \bigl( \bigl(
\operatorname{Bin}(n-1; V)-1 \bigr)_+,U_d \bigr),
\]
where the inner and outer binomials are independent. Analogous
expressions hold true for the remaining quantities.

Substituting \eqref{recHV} and \eqref{recVH} into each other gives
%
\begin{eqnarray}
\label{recVV} C_n^= (s) &\eqdist&1  + \I{s<U} \biggl[
\I{L_n>0} +C^{(=,1)}_{I_{=,1}^{(n)}}\biggl(\frac s U\biggr) +
C^{(=,
2)}_{I_{=,2}^{(n)}}\biggl(\frac s U\biggr)\biggr]
\nonumber
\\[-8pt]
\\[-8pt]
\nonumber
&&{}+ \I{s \geq U} \biggl[\I{R_n>0}+C^{(=, 3)}_{I_{=,3}^{(n)}}
\biggl(\frac
{s-U}{1-U}\biggr) + C^{(=, 4)}_{I_{=,4}^{(n)}}\biggl(
\frac{s-U}{1-U}\biggr)\biggr]
\end{eqnarray}
and
%
\begin{eqnarray}
\label{recHH} C_n^\perp(s) &\eqdist& 1  +
\I{S_n > 0} + \I{N_n > 0} + \I{s<U_d}
C^{(\perp, 1)}_{I_{\perp,1}^{(n)}}\biggl(\frac{s} {U_d}\biggr)\nonumber\\
&&{} + \I{s <
U_u} C^{(\perp, 2)}_{I_{\perp,2}^{(n)}}\biggl(\frac{s}{ U_u}\biggr)
\\
& &{}+\I{s \geq U_d} C^{(\perp, 3)}_{I_{\perp,3}^{(n)}}\biggl(
\frac{s-U_d} {
1-U_d}\biggr) + \I{s\ge U_u} C^{(\perp, 4)}_{I_{\perp,4}^{(n)}}
\biggl(\frac
{s-U_u}{1- U_u}\biggr),\nonumber
\end{eqnarray}
where $(C^{(=, i)}_n)_{n\geq0}$, $i=1,\ldots, 4$, are independent
copies of $(C^=_n)_{n\geq0}$, which are also independent of the family
$(U, I_{=,1}^{(n)}, I_{=,2}^{(n)}, I_{=,3}^{(n)}, I_{=,4}^{(n)})$ in
\eqref{recVV}, and
$(C^{(\perp, i)}_n)_{n\geq0}$, $i=1,\ldots, 4$, are
independent copies of $(C^\perp_n)_{n\geq0}$, which are also
independent of $(U_d, U_u, I_{\perp,1}^{(n)}, I_{\perp,2}^{(n)},
I_{\perp,3}^{(n)}, I_{\perp,4}^{(n)})$ in \eqref{recHH}.
Asymptotically, any limit $Z^=(s)$ of $n^{-\beta} C_n^=(s)$ should
satisfy the following fixed-point equation:
%
\begin{eqnarray}
\label{fixpV} Z^=(s)&\eqdist &\I{s<U}\biggl[(UV_\ell)^\beta
Z^{(=,1)}\biggl(\frac s U\biggr)+ \bigl(U(1-V_\ell )
\bigr)^\beta Z^{(=,2)}\biggl(\frac s U\biggr)\biggr]
\nonumber
\\
&&{}+\I{s\ge U} \biggl[ \bigl((1-U)V_r \bigr)^\beta
Z^{(=,3)} \biggl(\frac{s-U}{1-U}\biggr)\\
&&\hspace*{48pt}{}+ \bigl((1-U)
(1-V_r) \bigr)^\beta Z^{(=,4)}\biggl(
\frac{s-U}{1-U}\biggr)\biggr],\nonumber
\end{eqnarray}
where $Z^{(=, i)}$, $i=1,\ldots, 4$, are independent copies of $Z^=$,
independent of
$(U, V_\ell, V_r)$. Likewise any limit of $n^{-\beta} C_n^\perp(s)$
should satisfy
%
\begin{eqnarray}
\label{fixpH} Z^\perp(s)&\eqdist &\I{s<U_d}
(U_dV)^\beta Z^{(\perp,1)}\biggl(\frac{s} {U_d}
\biggr) + \I{s<U_u} \bigl(U_u(1-V) \bigr)^\beta
Z^{(\perp,2)}\biggl(\frac{s} {U_u}\biggr)
\nonumber
\\
&&{}+\I{s \geq U_d} \bigl((1-U_d)V
\bigr)^{\beta} Z^{(\perp, 3)} \biggl(\frac{s-U_d} {
1-U_d}\biggr)
\\
&&{}+\I{s \geq U_u} \bigl((1-U_u) (1-V)
\bigr)^{\beta}Z^{(\perp, 4)}\biggl(\frac
{s-U_u}{1- U_u}\biggr),\nonumber
\end{eqnarray}
where $Z^{(\perp, i)}$, $i=1,\ldots, 4$, are independent copies of
$Z^\perp$, independent of
$(U_d, U_u, V).$
Moreover, according to \eqref{recHV} and \eqref{recVH}, we expect a
connection between these two limits. This will be stated in the first
result of the next section and always allows us to focus on $C_n^=(s)$
first. The result for $C_n^\perp$ can then be deduced easily afterwards.

\subsection{About the conditions to use the contraction argument}
\textsc{Existence of continuous limit processes}. As in the case of quadtrees, one of the first steps consists of showing
the existence of the limit processes $Z^\perp$ and~$Z^=$.

\begin{prop} \label{existHV}
There exist two random continuous processes $Z^=, Z^\perp$ with $\Ec
{Z^=(s)} = \Ec{Z^H(s)} = h(s)$, finite absolute moments of all orders
such that $Z^=$ satisfies \eqref{fixpV} and $Z^\perp$ satisfies \eqref
{fixpH}. The laws of $Z^=$ and $Z^\perp$ are both unique under these
constraints.
Additionally:
\begin{itemize}
\item
%
\begin{equation}
\label{fixpHV} \frac{2}{\beta+1} Z^\perp(s) \eqdist
V^{\beta}Z^{(=, 1)} (s) + (1- V)^{\beta}Z^{(=, 2)} (s)
\end{equation}
and
\[
\frac{\beta+1}{2} Z^=(s) \eqdist \I{s<U} U^{\beta}Z^{(\perp, 1)}
\biggl(\frac s U\biggr) + \I{s \geq U} (1-U)^{\beta}Z^{(\perp,
2)}
\biggl( \frac{s-U} {1-U}\biggr).
\]

\item
For every fixed $s\in[0,1]$, $Z^=(s)$ is distributed like $Z(s)$ where
$Z$ is the process constructed in Section~\ref{seclimit}. In
particular, $\V{Z^=(s)}$ is given in \eqref{constvarfix} and $\V
{Z^\perp(s)} = K_2^\perp h^2(s)$, where
%
\begin{equation}
\label{varH} K_2^\perp= \biggl( \frac{2 c_2}{2\beta+1} \biggl(
\frac{\beta
+1}{2} \biggr)^2 + 2 \mathrm{B} (\beta+1, \beta+1 )
\biggl( \frac{\beta
+1}{2} \biggr)^2-1 \biggr),
\end{equation}
and $c_2$ is defined in \eqref{eqKbar}.

\item If $\xi$ is uniform on $[0,1]$ and independent of $Z^=,Z^\perp$,
then $\V{Z^=(\xi)}=\V{Z(\xi)}$ and
%
\begin{eqnarray}
\label{constvarH}&& \operatorname{Var} \bigl(Z^\perp(\xi) \bigr)\nonumber\\
&&\qquad =
K_3^\perp = \biggl( \frac{2 c_2 }{2\beta+1} + 2 \mathrm{B} (
\beta+1, \beta+1 ) \biggr) \biggl( \frac{\beta
+1}{2} \biggr)^2
\mathrm{B}(\beta+1,\beta+1)
\\
&&\hspace*{30pt}\qquad\quad{} - \biggl(\mathrm{B} \biggl( \frac{\beta}{2} +1, \frac{\beta}{2} +1
\biggr) \biggr)^2.\nonumber
\end{eqnarray}
\end{itemize}
\end{prop}
\begin{pf}
The fixed-point equation \eqref{fixpV} is very similar to that in
\eqref{eqlimitprocess}, and we use the approach that has proved
fruitful in Section~\ref{seclimit}. More precisely, the construction
of $Z(s)$ slightly modified to $Z^=(s)$. Define the operator
$G^=:[0,1]^3 \times\mathcal C[0,1]^4 \to\mathcal C[0,1]$ by
\begin{eqnarray*}
&&G^=(x,y,z, f_1,f_2,f_3,f_4) (s)
\\
&&\qquad=\I{s<x} \biggl[(xy)^\beta f_1\biggl(\frac sx\biggr)+
\bigl(x(1-y) \bigr)^\beta f_2\biggl(\frac sx\biggr) \biggr]
\\
&&\qquad\quad{} +\I{s\ge x} \biggl[ \bigl((1-x)z \bigr)^\beta f_3
\biggl( \frac{s-x}{1-x} \biggr)+ \bigl((1-x) (1-z) \bigr)^\beta
f_4 \biggl(\frac{s-x}{1-x} \biggr) \biggr].
\end{eqnarray*}
Then let (as in Section~\ref{seclimit})
\[
Z_{n+1}^{=,u}=G^= \bigl(U_u, V_u,
W_u, Z_n^{=, u1}, Z_n^{=, u2},
Z_n^{=, u3}, Z_n^{=,u4} \bigr),\qquad
Z_0^{=,u} = h(s)
\]
for all $u \in\mathcal T$, where $\{U_v, v\in\mathcal T\}, \{V_v,
v\in\mathcal T\}$
and $\{W_v, v\in\mathcal T\}$ are three independent families of
i.i.d. $[0,1]$-uniform random variables.
Lemma \ref{lemboundfixeds} remains true for $Z_n^=:= Z_n^{=,
\varnothing}$ since $W_n^=$ equals $W_n$ in distribution where $W_n$
appears in \eqref{tube}.
Since also $L_n^=$ and $L_n$ (appearing in Lemma~\ref
{lemclosestpair}) coincide in distribution, \eqref
{eqXborel-cantelli} holds true for
$Z_n^=$ and therefore Proposition~\ref{propcontinuouspaths} remains valid.
The existence of all moments of $\sup_{s \in[0,1]} Z^=(s)$ follows in
the same way. Finally, note that
$Z^=_n(s)$ is distributed as $Z_n(s)$ for all fixed $n,s$, hence the
one-dimensional distributions of $Z^=$ and $Z$ coincide.
It is now easy to see that $Z^\perp$ defined by \eqref{fixpHV} solves
\eqref{fixpH}.
The uniqueness of $Z^=(s)$ [resp., $Z^=(s)$] follows by contraction with
respect to the $\zeta_2$ metric; compare Lemma 18 in \cite{NeSu2011a}.
Finally, the variance of $Z^\perp(s)$ can be computed as in Section~\ref
{secmomentssup} but it
is much easier to use \eqref{fixpHV}, we omit the calculations.
\end{pf}

\textsc{Uniform convergence of the mean.}
Comparing construction and recurrence for partial match queries in
$2$-d trees and quadtrees it seems very likely that this quantities are
not only
of the same asymptotic order in the case of a uniform query but also
closely related for fixed $s \in[0,1]$ and $n \in\N$. This can be
formalized by the following lemma:
%
\begin{lem} \label{lemcompareE}
For any $s \in[0,1]$ and $n \in\N$, we have
\[
\tfrac{1}{5} \mathbf{E} \bigl[C_n(s)\bigr] \leq \mathbf{E}
\bigl[C_n^{=}(s)\bigr] \leq2 \mathbf{E}
\bigl[C_n(s)\bigr].
\]
\end{lem}
\begin{pf}
We prove both bounds by induction on $n$ using the recursive
decompositions \eqref{eqCnrec}, \eqref{recVV}. Both inequalities are
obviously true for
$n = 0,1$.
Assume that the assertions were true for all $m \leq n-1$ and $s \in
[0,1]$. We start with the upper bound which is easier. By \eqref
{recVV}, we have
\begin{eqnarray*}
\mathbf{E} \bigl[C_n^= (s) \bigr] &\leq&2  + \mathbf{E} \biggl[
\I{s<U} \biggl[C^{(=,1)}_{I_{=,1}^{(n)}} \biggl(\frac s U\biggr) +
C^{(=, 2)}_{I_{=,2}^{(n)}}\biggl(\frac s U\biggr)\biggr] \biggr]\\
&&{}+ \mathbf{E} \biggl[\I{s \geq U} \biggl[C^{(=, 3)}_{I_{=,3}^{(n)}}
\biggl( \frac
{s-U}{1-U}\biggr) + C^{(=, 4)}_{I_{=,4}^{(n)}}\biggl(
\frac{s-U}{1-U}\biggr)\biggr] \biggr].
\end{eqnarray*}
Hence, it suffices to show that
\[
\mathbf{E} \biggl[\I{s < U} C^{(=,1)}_{I_{=,1}^{(n)}}\biggl(\frac s U
\biggr) \biggr] \leq2 \mathbf{E} \biggl[\I {s < U} C^{(1)}_{I_1^{(n)}}
\biggl( \frac s U \biggr) \biggr].
\]
This can be done in two steps. First, by conditioning on
$I_{=,1}^{(n)}$ and $U$, using the induction hypothesis, we have
\[
\mathbf{E} \biggl[\I{s < U} C^{(=,1)}_{I^{(n)}_{=,1}} \biggl( \frac s U
\biggr) \biggr] \leq2 \mathbf{E} \biggl[\I{s < U} C^{(1)}_{I^{(n)}_{=,1}}
\biggl( \frac s U \biggr) \biggr].
\]
Finally, conditioning on $U$, $I^{(n)}_{=,1}$
is stochastically smaller than $I^{(n)}_1$ which gives
\[
\mathbf{E} \biggl[\I{s < U} C^{(1)}_{I^{(n)}_{=,1}} \biggl( \frac s U
\biggr) \biggr] \leq2 \mathbf{E} \biggl[\I{s < U} C^{(1)}_{I_1^{(n)}}
\biggl( \frac s U \biggr) \biggr]
\]
by monotonicity of $n \to\E{C_n(s)}$. For the lower bound, note that
\begin{eqnarray*}
\mathbf{E} \bigl[C_n^= (s) \bigr]& \geq&1  + \mathbf{E} \biggl[
\I{s<U} \biggl[C^{(=,1)}_{I_{=,1}^{(n)}} \biggl(\frac s U\biggr) +
C^{(=, 2)}_{I_{=,2}^{(n)}}\biggl(\frac s U\biggr)\biggr] \biggr]
\\
&&{}+ \mathbf{E} \biggl[\I{s \geq U}\biggl[C^{(=, 3)}_{I_{=,3}^{(n)}}\biggl(
\frac
{s-U}{1-U}\biggr) + C^{(=, 4)}_{I_{=,4}^{(n)}}\biggl(
\frac{s-U}{1-U}\biggr)\biggr] \biggr].
\end{eqnarray*}
Therefore, it is enough to prove
\[
\mathbf{E} \biggl[\I{s < U} C^{(=,1)}_{I_{=,1}^{(n)}}\biggl(\frac s U
\biggr) \biggr] \geq \frac
{1}{5} \biggl( \mathbf{E} \biggl[\I{s < U}
C^{(1)}_{I_1^{(n)}} \biggl( \frac s U \biggr) \biggr] - 1 \biggr).
\]
This can be done as for the upper bound. First, by the induction
hypothesis, we have
\[
\mathbf{E} \biggl[\I{s < U} C^{(=,1)}_{I^{(n)}_{=,1}} \biggl( \frac s U
\biggr) \biggr] \geq \frac{1}{5} \mathbf{E} \biggl[\I{s < U}
C^{(1)}_{I^{(n)}_{=,1}} \biggl( \frac s U \biggr) \biggr].
\]
The result follows as for the upper bound by the fact that
$I_{=,1}^{(n)}$ is stochastically larger than $(I_1^{(n)}-1)^+$ and
$C^{(1)}_{(I_1^{(n)}-1)^+} \geq C^{(1)}_{I_1^{(n)}} -1$.
\end{pf}
Recalling \eqref{eqk1} and \eqref{eqk2}, it is natural to introduce the
constants
\begin{eqnarray}
\label{constK1HV} K_1^= = \frac{\kappa_=}{\mathrm{B} ({\beta}/{2} +1, {\beta
}/{2} +1  )},\qquad K_1^\perp=
\frac{\kappa_\perp}{\mathrm{B} (
{\beta}/{2} +1, {\beta}/{2} +1  )}
\nonumber
\\[-8pt]
\\[-8pt]
\eqntext{\mbox{with }\displaystyle K_1^\perp=
\frac{2}{1+\beta} K_1^=,}
\end{eqnarray}
and the functions $\mu_1^\perp(s) = K_1^\perp h(s)$, and $\mu_1^=(s) =
K_1^= h(s).$
%
\begin{prop}
There exists $\varepsilon_= > 0$ such that
\[
\sup_{s \in[0,1]} \bigl|n^{-\beta} \mathbf{E} \bigl[C_n^=(s)
\bigr] - \mu_1^=(s)\bigr| = O \bigl(n^{-\varepsilon_=} \bigr),
\]
and the analogous result holds true for $\Ec{C_n^\perp(s)}$.
\end{prop}
We proceed as in Section~\ref{secunifconvergence} by considering the
continuous-time process $P_t^=(s)$. Since we have already proved an
analogous result for the case of quadtree, we give a brief sketch that
focuses on the few locations where the arguments have to be modified.

\begin{pf*}{Sketch of proof}
The first step is to prove point-wise convergence
which is done as Curien and Joseph \cite{CuJo2010}. By Lemma~\ref{lemcompareE}, using
a Poisson$(t)$ number of points, we have
%
\begin{equation}
\label{compareEt} \tfrac1 5 \mathbf{E} \bigl[P_t(s)\bigr] \leq
\mathbf{E} \bigl[P^=_t(s)\bigr] \leq2 \mathbf{E}
\bigl[P_t(s)\bigr].
\end{equation}
Let $\tau_1^=$ be the arrival time of the first point which yields a
partitioning line that intersects the query line $\{x = s \}$, and let
$Q_1^= = Q_1^=(s)$ be the lower of the two rectangles created by this
cut (for the expected value we are about to compute, they both look the
same). Let $\xi_1^=:= \xi_1^=(s)$ be the relative position of the
query line $s$ within the rectangle $Q_1^=$ and $M_1^= = \operatorname{Leb}(Q_1^=)$.
Then, denoting $\tau$ the arrival time of the first point in the
process, we have
\[
\mathbf{E} \bigl[P_t^=(s)\bigr] = \mathbf{P} (t \geq\tau ) +
\mathbf{P} \bigl(t \geq \tau_1^= \bigr) + 2 \mathbf{E} \bigl[\tilde
P^=_{M_1^= t-\tau_1^=} \bigl(\xi_1^= \bigr)\bigr],
\]
where $(\tilde P^=(t))_{t \geq0}$ denotes an independent copy of
$(P^=(t))_{t \geq0}$ and $\tilde P^=(t) = 0$ for $t < 0$.
Similarly, let $\tau_k^=$ be the arrival time of the first point which
cuts $Q_{k-1}^=$ perpendicularly to the query line. Let
$Q^=_k$ be the lower of the two rectangles created by this cut, and let
$\xi_k^=$ be the position of the query line $s$ relative to the rectangle
$Q_k^=$. With this notation and $M_k^= = \operatorname{Leb}(Q_k^=)$, we have
\[
\mathbf{E} \bigl[P_t^=(s)\bigr] = g_k^=(t)+2^k
\mathbf{E} \bigl[\tilde P^=_{M_k^= t-\tau_k^=} \bigl(\xi_k^= \bigr)
\bigr],
\]
where $0 \leq g_k^=(t) \leq2^{k+1}$.\eject

We need to modify the inter-arrival times ${\zeta'}_k^{=} = \tau_k^= -
\tau_{k-1}^=$. We can split ${\zeta'}_k^{=}$ in the time it takes for the
first vertical point to fall in $Q_{k-1}^=$ which we denote by ${\zeta
'}_k^{=,1}$ and the remaining time by
${\zeta'}_k^{=,2}$. Letting $M_k^= = \vol(Q_k^=)$, the normalized
versions of the inter-arrival times with unit mean are
\begin{eqnarray*}
\zeta_k^{=,1} & =& {\zeta'}_k^{=,1}
\cdot M_{k-1}^=,
\\
\zeta_k^{=,2} & =& \biggl( \frac{\xi^=_k}{\xi^=_{k-1}} \I{
\xi^=_k < \xi ^=_{k-1}} + \frac{\xi^=_{k-1}}{\xi^=_k} \I{
\xi^=_k \geq\xi^=_{k-1}} \biggr) {\zeta'}_k^{=,2}
\cdot M_{k-1}^= \geq{\zeta'}_k^{=,2}
\cdot M_{k-1}^=.
\end{eqnarray*}
Write $\mathcal M_k=M_k/M_{k-1}$. Observe that, given $\mathcal M_0^=,
\ldots, \mathcal M_k^=$, the random variable $F_k^= = M_k^= \cdot\tau
_k^=$ is not independent of $(\xi_{\ell})_{0 \leq\ell\leq k}$, a
property which is used in \cite{CuJo2010} and in the proof of Lemma~\ref
{lemuniformmiddle} in the present paper. However we can
use the trivial lower bound $0 \leq F_k$ and the upper bound obtained
by bounding ${\zeta'}_k^{=,2}$ from above by $\zeta_k^{=,2} / M_{k-1}^=$.
Then, using almost sure monotonicity of $P_t(s)$ (in $t$) and \eqref
{compareEt} to transform bounds for the mean in the quadtree to bounds
in the
$2$-d tree (and vice versa), it is easy to see that the techniques of
Section 4 in \cite{CuJo2010} work equally well in this case. The limit
$\mu_1^=(s)$ is identified as in Section 5 of \cite{CuJo2010}
since both limits satisfy the same fixed-point equation.

The generalization to uniform convergence with polynomial rate can be
worked out as in
Section~\ref{secunifconvergence} (of the present document) using the
modifications we have described above. The constants appearing in the
course of Section~\ref{secunifconvergence} need to be modified, but
$\varepsilon_=$ may be chosen to equal the value of $\varepsilon$ in
Proposition~\ref{propuniformpoisson}.
The de-Poissonization is routine, and we omit the details.

Finally, we indicate how to proceed with $\Ec{C_n^\perp(s)}$. The
arguments above can be used to treat prove uniform convergence of
$n^{-\beta} \Ec{C_n^\perp(s)}$ on $[0,1]$; we present a direct approach
relying on \eqref{recHV}. We have
\begin{eqnarray*}
n^{-\beta} \mathbf{E} \bigl[C_n^\perp(s)\bigr] &=&
n^{-\beta} + 2 n^{-\beta} \mathbf{E} \bigl[C^=_{S_n} (s)
\bigr]
\\
& =& n^{-\beta} + 2 \int_0^1 \sum
_{k=0}^{n-1} \bigl(\mu_1^=(s) +
O \bigl(k^{-\varepsilon_=} \bigr) \bigr) \frac{k^\beta}{n^\beta} \mathbf{P} \bigl(
\operatorname {Bin}(n-1,v)=k \bigr) \,dv
\\
& =& n^{-\beta} + 2 \mu_1^=(s) \cdot\frac{
\E{\operatorname{Bin}(n-1,V)^\beta
}}{n^\beta} \\
&&{}+ O
\bigl(n^{-\beta} \mathbf{E} \bigl[\operatorname{Bin}(n-1,V)^{\beta- \varepsilon
_=}
\bigr] \bigr)
\\
& =& \mu_1^\perp(s) + O \bigl(n^{-\varepsilon_=} \bigr),
\end{eqnarray*}
uniformly in $s \in[0,1]$ using Minkowski's inequality, the
concentration for binomial in \eqref{eqalphamomentbin}, and \eqref
{constK1HV} for the first term and Jensen's inequality for the second.
\end{pf*}

\subsection{The limiting behavior in $2$-d trees}
We are finally ready to state the version of our main result for $2$-d
trees. It is proved along the same lines we used for the case of
quadtrees, and we omit the details.

\begin{thmm} \label{thmmainVH}
With the processes $Z^=$ and $Z^\perp$ of Proposition \ref{existHV} we have
\[
\biggl( \frac{C_n^=(s)}{K_1^= n^{\beta}} \biggr)_{s \in[0,1]} \rightarrow \bigl(Z^=(s)
\bigr)_{s \in[0,1]},\qquad\biggl( \frac
{C_n^\perp(s)}{K_1^\perp n^{\beta}} \biggr)_{s \in[0,1]}
\rightarrow \bigl( Z^\perp(s) \bigr)_{s \in[0,1]},
\]
in distribution in $\Do$ endowed with the Skorokhod topology. Here
$K_1^=$ and $K_1^\perp$ are defined in \eqref{constK1HV}.
For $s\in[0,1]$
\[
n^{-\beta} \mathbf{E} \bigl[C_n^=(s)\bigr] \rightarrow
K_1^= h(s),\qquad n^{-2 \beta} \operatorname{Var}
\bigl(C^=_n(s) \bigr) \rightarrow \bigl(K_1^=
\bigr)^2 K_2 h(s)^2
\]
and
\[
n^{-\beta} \mathbf{E} \bigl[C_n^\perp(s)\bigr]
\rightarrow K_1^\perp h(s),\qquad  n^{-2
\beta}
\operatorname{Var} \bigl(C^\perp_n(s) \bigr) \rightarrow
\bigl(K_1^\perp \bigr)^2 K_2^\perp
h(s)^2,
\]
where $K_2$ is given in \eqref{constvarfix} and $K_2^\perp$ in \eqref{varH}.

If $\xi$ is uniformly distributed on $[0,1]$, independent of
$(C_n^=)_{n\geq0}, (C_n^\perp)_{n \geq0}$ and $Z^=, Z^\perp$,
then
\[
\frac{C^=_n(\xi)}{K_1^= n^{\beta}} \stackrel{d} {\longrightarrow} Z^=(\xi ),\qquad \frac{C^\perp_n(\xi)}{K_1^\perp n^{\beta}}
\stackrel {d} {\longrightarrow} Z^\perp(\xi),
\]
with convergence of the first two moments in both cases.
In particular
\[
\operatorname{Var} \bigl(C^=_n(\xi) \bigr) \sim K_4^=
n^{2 \beta}, \qquad\operatorname{Var} \bigl(C^\perp_n(\xi)
\bigr) \sim K_4^\perp n^{2 \beta},
\]
where $K_4^= = (K_1^=)^2 K_3 \approx0.69848$, $K_4^\perp= (K_1^\perp
)^2 K^V_3 \approx0.77754,$ with
$K_3=\V{Z(\xi)}$ in \eqref{eqvarzxi} and $K_3^\perp$ in \eqref{constvarH}.
\end{thmm}

Note that since $Z^=(s)$ equals $Z(s)$ in distribution for fixed $s \in
[0,1]$ we can characterize $Z^=(s)$ as in \eqref{1dimrv}.
\eqref{fixpHV} together with Proposition \ref{existHV} implies that
for fixed $s \in[0,1]$
\[
Z^\perp(s) \eqdist\Psi^\perp\cdot h(s) \qquad\mbox{with } \Psi
^\perp= \frac{\beta+1}{2} \bigl(V^\beta\Psi+
(1-V)^\beta\Psi' \bigr),
\]
where $\Psi'$ is an independent copy of $\Psi$, $\Psi$ being defined in
Theorem~\ref{thmsame-Xi} and $V$ is independent of $(\Psi,\Psi')$. In
particular, we have
\[
\mathbf{E} \bigl[ \bigl(\Psi^\perp \bigr)^m \bigr] =
\biggl( \frac{\beta+1}{2} \biggr)^m \sum
_{\ell=0}^m \pmatrix{m \cr\ell} \mathrm{B} \bigl(\beta
\ell+1, \beta(m-\ell) +1 \bigr) c_\ell c_{m-\ell}
\]
for $m \geq2$ where $c_m = \Ec{\Psi^m}$ satisfies recursion \eqref
{recmomentsZ} and $c_0 = c_1 = 1$. \par
Also, as in the quadtree case, it is possible to give convergence of
mixed moments of arbitrary order, compare Corollary \ref{thmmixedn}, and
distributional and moment convergence of the suprema of the processes
after rescaling as in Theorem \ref{thmsupremum}.

\begin{appendix}\label{app}
\section{About the geometry of random quadtrees}

\begin{lem} \label{lemapp1}
Let $W_n$ denote the maximum width of a cell at level $n$ {in the
construction of $Z_n$ and $c < 1$}. Then
\[
\mathbf{P}\bigl(W_n \ge c^n\bigr)
\le \bigl(4e\log(1/c)
\bigr)^n.
\]
\end{lem}
\begin{pf}
Let $U_i$, $i\ge1$ be a family of i.i.d. $[0,1]$-uniform random
variables and $E_i$, $i\ge1$, be a family of i.i.d. $\operatorname{exponential}(1)$
random variables. Then, the union bound and a large deviations argument yields
\begin{eqnarray*}
\mathbf{P}\bigl(W_n \ge c^n\bigr)
 &\le&4^n \cdot
\mathbf{P}\bigl(U_1 \cdot U_2 \cdots U_n \ge
c^n\bigr)
\\
&=& 4^n \cdot
\mathbf{P}\Biggl(\sum_{i=1}^n
E_i \le n \log(1/c)\Biggr)
\\
&\le&4^n \exp \bigl(-n \bigl(\log(1/c)-1-\log\log(1/c) \bigr) \bigr)
\\
&\le& \bigl(4e\log(1/c) \bigr)^n
\end{eqnarray*}
as desired.
\end{pf}

\begin{lem}\label{lemfill-up}Let $F_k$ be the fill-up level of a
random quadtree of size $k$. Then, for every integer number {$x> 22$}
there exists an integer $n_0(x)$ with
\[
\mathbf{P}(F_{x^n}< n )
\le4^{n+1} x^{-n/100},\qquad n \geq
n_0(x).
\]
\end{lem}
\begin{pf}We consider the $4^n$ possible nodes in level $n$. By
symmetry each of them is occupied by a key with the same probability.
Looking at a specific one, for example, the leftmost, size of the
corresponding subtree is stochastically bounded by $\operatorname{Bin}(x^n;
U_1V_1\cdots U_nV_n)-n$ where $\{U_i, i\ge1\}$ and $\{V_i, i\ge1\}$
are independent families of i.i.d. $[0,1]$-uniform random variables.
Then by the union bound applied to the $4^n$ cells at level $n$, using
Chernoff's inequality,
we have
%
\begin{eqnarray}
\label{eqfill-up}
\mathbf{P}(F_{x^n}<n)
 &\le&4^n \cdot
\mathbf{P}\bigl(\operatorname{Bin}\bigl(x^n; U_1V_1\cdots
U_nV_n\bigr)\leq n\bigr)
\nonumber
\\
&\le&4^n \cdot\exp\bigl(- \bigl(1-n2^{-n}
\bigr)^2 2^{n+1}\bigr) \\
&&{}+ 4^n
\mathbf{P}\biggl(U_1V_1 \cdots U_nV_n \le
\biggl( \frac{2}{x} \biggr)^n\biggr).\nonumber
\end{eqnarray}
However, using once again the large deviations principle for sums of
i.i.d. exponential random variables $E_i, i\ge1$,
%
\begin{eqnarray}
\label{eqproductunif}
\mathbf{P}\bigl(U_1V_1\ldots
U_nV_n\le(2/x)^n\bigr)
 & = &
\mathbf{P}\Biggl(\sum
_{i=1}^{2n} E_i \ge n \log(x/2)\Biggr)
\nonumber
\\
& \le&\exp\biggl(-2n \biggl( \frac{\log(x/2)}{2}-1-\log\frac{\log
(x/2)}{2} \biggr)
\biggr)
\\
& \le& x^{-n/100}\nonumber
\end{eqnarray}
for all $x > 22$ since then $\frac{e^2}{2} \log^2(x/2) \leq x^{99/100}$.
Putting (\ref{eqfill-up}) and (\ref{eqproductunif}), we obtain
\[
\mathbf{P}(F_{x^n}<n)  \le 4^n \exp \bigl(-2^{n-1}
\bigr) + 4^n \cdot x^{-n/100} \le4^{n+1}
x^{-n/100}
\]
for $x > 22$ and $n$ large enough.
\end{pf}

\begin{lem}\label{lemclosestpair} {There exists $0 < \gamma_0 < 1$
such that any positive real number $\gamma< \gamma_0$}, there exists
an integer $n_1(\gamma)$ with
\[
\mathbf{P}\bigl(L_n< \gamma^n\bigr)
 \le 6^{n+1}
\gamma^{n/201}, \qquad n \geq n_1(\gamma).
\]
\end{lem}
\begin{pf}The joint distribution of the $x$-coordinates of the
vertical lines in the tree developed up to level $n$ is complex. In
particular, it is \emph{not} that of independent uniform points on
$[0,1]$. However, we can use a simple coupling with a family of i.i.d.
random points on $[0,1]^2$ that yields a good enough lower bound on $L_n$.

Let $\xi_i=(U_i,V_i)$, $i\ge1$ be i.i.d. uniform random points on
$[0,1]^2$. Let $T_k$ be the quadtree obtained by inserting the random
points $\xi_i$, $1\le i\le k$, in this order. Write $D_i$ for the depth
at which the point $\xi_i$ is inserted; so, for instance, $D_1=0$. Let
$K_n$ be the first $k$ for which the tree $T_k$ is complete up to level
$n$; we mean here that $T_{k}$ should have $4^n$ cells at level $n$, so
it should have $4^{n-1}$ nodes at level $n-1$. Then, by definition $\{
\xi_i\dvtx i\ge1, D_i< n\}$ has the distribution of the set of points used
to construct the process $Z_n$. Obviously, $\{\xi_i\dvtx i\ge1, D_i< n\}
\subseteq\{\xi_i\dvtx 1\le i\le K_n\},$ and for any natural number $x>0$,
\begin{eqnarray*}
\mathbf{P}\bigl(L_n< \gamma^n\bigr)
 &\le&
\mathbf{P}\bigl(\exists i,j\le K_n
\dvtx i\neq j, |U_i-U_j|<\gamma^n\bigr)
\\
&\le&
\mathbf{P}\bigl(\exists i,j\le x^n\dvtx i\neq j, |U_i-U_j|<
\gamma^n\bigr)
 +
\mathbf{P}\bigl(K_n>x^n\bigr)
\\
&\le& x^{2n}\cdot2\gamma^n +
\mathbf{P}\bigl(K_n >
x^n\bigr),
\end{eqnarray*}
by the union bound.
The random variable $K_n$ is related to the fill-up level of a random
quadtree, which has been studied by \cite{Devroye1987}; see also \cite
{DeLa1990}. We could not find a reference giving a precise tail bound,
so we proved one here in Lemma~\ref{lemfill-up}.
We obtain
\[
\mathbf{P}\bigl(K_n > x^n\bigr)
 =
\mathbf{P}(F_{x^n}< n)
\le4
\bigl(4x^{-1/100} \bigr)^n
\]
as long as $x\ge22$ and $n \ge n_0(x)$ (the condition for the bound in
Lemma~\ref{lemfill-up} to hold).
It follows readily that
\begin{eqnarray*}
\mathbf{P}\bigl(L_n<\gamma^n\bigr)
 &\le&2 \bigl(x^2\gamma
\bigr)^n + 4 \bigl(4x^{-1/100} \bigr)^n
\\
&\le&6^{n+1} \gamma^{n/201},
\end{eqnarray*}
upon choosing { $x = \lceil4^{100/201} \gamma^{-100/201}\rceil$ (i.e.,
$x^2\gamma\approx 4 x^{-1/100}$) and $\gamma< 4 \cdot{22^{-2.01}}$ which
implies $ x > 22$.}
This completes the proof.
\end{pf}

\section{\texorpdfstring{Complements to the proof of Propopsition~\protect\ref{PROPUNIFORMCONVERGENCE}}
{Complements to the proof of Propopsition 12}}

\subsection{\texorpdfstring{Behavior away from the edge: Proof of Lemma~\protect\ref{lemuniformmiddle}}
{Behavior away from the edge: Proof of Lemma 15}} \label{subsecaway}

The core of the work is to bound the second term in \eqref
{equniformdecomp} involving $s\in(\delta,1/2]$. We prove that $\Ec
{P_t(s)}$ is uniformly Cauchy on $(\delta,1/2]$ by tightening some of
the arguments in~\cite{CuJo2010}. We could start from (14) there, but
we feel that the reader would follow more easily if we re-explain the
approach. Observe that most of the quantities defined in the remaining
of the section will depend on $s$ which we will neglect in the notation
for the sake of readability.

The first step is to unfold $k$ levels of the fundamental recurrence
\eqref{eqCnrec} in the Poisson case. Let $\tau_1$ be the arrival time
of the first point in
the Poisson process and $Q_1 = Q_1(s)$ be the lower of the two
rectangles that intersect the line $\{x=s\}$ after inserting the first point.
Inductively let $\tau_k = \tau_k(s)$ be the arrival time of the first
point of the process in the region $Q_{k-1}$ and $Q_k$ be the lower of
the two rectangles
that hit the line $\{x=s\}$ at time $\tau_k$. For convenience, set $Q_0
= [0,1]^2$. Finally, let $\tilde P_t$ be an independent copy of the
process $P_t$
(set $\tilde P_t \equiv0$ for $t < 0$).
At level one, using the horizontal symmetry, we have
\[
\mathbf{E} \bigl[P_t(s)\bigr] =
\mathbf{P}(t\ge\tau_1)
+2
\mathbf{E} \bigl[\tilde P_{\vol(Q_1)(t-\tau_1)}(\xi_1)\bigr],
\]
where $\xi_1 = \xi_1(s)\in[0,1]$ denotes the location of the line $\{
x=s\}$ relative to the region $Q_1$. If the interval $[\ell_1,r_1]$
denotes the projection of
$Q_1$ on the first axis, we have
\[
\xi_1(s)=\frac{s-\ell_1}{r_1-\ell_1}.
\]
Write $\xi_k = \xi_k(s)\in[0,1]$ for the location of the line $\{x=s\}
$ relatively to the region $Q_k$, and $M_k=\vol(Q_k)$. Then, unfolding
up to level $k$, we obtain
%
\begin{equation}
\label{eqkunfold0} \mathbf{E} \bigl[P_t(s)\bigr]
=g_k(t) + 2^k \mathbf{E} \bigl[\tilde
P_{M_k (t-\tau_k)}(\xi_k)\bigr],
\end{equation}
where $0\le g_k(t)\le2^k-1$. Next, we introduce the inter-arrival
times $\zeta_k' = \tau_k - \tau_{k-1}$ with $\zeta_0':= 0$ and their
normalized versions $\zeta_k = \zeta_k' M_{k-1}$ (again $\zeta_0:=
0$). Defining $F_k = M_k \tau_k$, we can rewrite (\ref{eqkunfold0}) as
%
\begin{equation}
\label{eqkunfold} \mathbf{E} \bigl[P_t(s)\bigr]
=g_k(t) + 2^k \mathbf{E} \bigl[\tilde
P_{M_k t-F_k}(\xi_k)\bigr].
\end{equation}
Note that $(\zeta_k)_{k \geq1}$ are i.i.d. exponential random
variables with unit mean, also independent of $(\xi_k, Q_k)_{k \geq
1}$. \par
Before going any further, note that, as we have already seen in
Section~\ref{seclimit}, the region $Q_k$, is not distributed like a
typical rectangle at
level $k$; in particular $\vol(Q_k)$ is not distributed as $X_1Y_1
\cdots X_kY_k$, for independent
$[0,1]$-uniform random variables $X_i,Y_i$, $i\ge1$. Intuitively,
$Q_k$ should be stochastically larger than a typical cell, since it is
conditioned to intersect the line $\{x=s\}$. This is verified by the
following lemma.
%
\begin{lem}\label{lemvollower}For any $s\in(0,1)$, any integer $k\ge
0$ and $1\le i\le2^k$, we have
\[
\vol(Q_k)= M_k \gest X_1Y_1
\cdots X_kY_k,
\]
where $X_i,Y_i$, $i\ge1$ are independent random variables uniform on $[0,1]$.
\end{lem}
\begin{pf}Consider one split, at a point $(X,Y)$ uniform inside the
unit square. The split creates four new boxes, two of them being hit by
$s$. Let $L$ be the length these two cells. Their height is either $Y$
or $(1-Y)$, which are both uniform. So it suffices to prove that $L
\gest X$. By symmetry, it suffices to consider $s\le1/2$. We have
\[
L=\I{s\le X}X + \I{s>X}(1-X).
\]
Write $F_L(y)=\p{L\le y}$ and $F_X(y)=\p{X\le y}=y$. It is then easy to
see that
\[
F_L(y)=
\mathbf{P}(L\le y)
= \cases{ %
 0,&\quad $y
\le s,$
\vspace*{2pt}\cr
y-s,&\quad $s\le y\le1-s,$
\vspace*{2pt}\cr
2y-1,& \quad $y\ge1-s.$}
\]
Hence, for all $s\in(0,1/2)$ and all $y\in(0,1)$ we have $F_L(y)\le
y=F_X(y)$. The result follows.
\end{pf}

The second term will be treated using results for the case $s = \xi$,
for a uniform random variable $\xi$ independent of everything else.
Curien and Joseph \cite{CuJo2010} found a very clever way to circumvent the problem that
for any $k \geq1$,
the random variable $\xi_k$ is not uniformly distributed on $[0,1]$. In
their Proposition 4.1 they introduce a version of the homogeneous
Markov chain
$(\xi_k, \mathcal M_k)_{k \geq1}$ where $\mathcal M_k:= M_k/M_{k-1}$
together with a random time $T$ such that for any $k \in\N$,
conditionally on
$\{T \leq k\}$, the random variable $\xi_k$ is uniformly distributed on
$[0,1]$, independent of $(\mathcal M_1, \ldots, \mathcal M_k, T)$.
Choosing these random variables
independent of the process $\tilde P_t$ we will use them in the
following without changing the notation [$F_k$ can be constructed using
$(\mathcal M_{\ell})_{1 \leq\ell\leq k}$
and an additional set of i.i.d. exponential random variables with mean one].
The details of the definition of $T$ are not important for us. The only
crucial thing is that $T$ has exponential tails. Indeed, we have page
15 of~\cite{CuJo2010},
%
\begin{equation}
\label{eqboundK} \mathbf{E} \bigl[1.15^T\bigr] \le C_4
\bigl(s\wedge(1-s) \bigr)^{-1/2} \leq C_4
\delta^{-1/2}
\end{equation}
for some constant $C_4$ in the present case, $ \delta< s \le1/2$.\eject

Then, using \eqref{eqkunfold} and the triangle inequality, we obtain
for any $t$ and $r$ such that $r\ge t$,
%
\begin{eqnarray}
\label{eqmainmiddle}&& \bigl|t^{-\beta} \mathbf{E} \bigl[P_t(s)
\bigr] -r^{-\beta} \mathbf{E} \bigl[P_r(s)\bigr] \bigr|\nonumber\\
&&\qquad
\le2^k\bigl |t^{-\beta} \mathbf{E} \bigl[\tilde P_{M_k t-F_k}(
\xi_k)\bigr] -r^{-\beta} \mathbf{E} \bigl[\tilde
P_{M_k r-F_k}(\xi_k)\bigr]\bigr | + 2^{k+1}
r^{-\beta}
\nonumber
\\
&&\qquad\le2^k \bigl|t^{-\beta} \mathbf{E} \bigl[\tilde
P_{M_k t-F_k}( \xi_k)\I{T\le k}\bigr] -r^{-\beta}
\mathbf{E} \bigl[\tilde P_{M_k r-F_k}( \xi_k)\I{T\le k}\bigr] \bigr|
\\
&&\qquad\quad{} + 2^k \bigl|t^{-\beta} \mathbf{E} \bigl[\tilde
P_{M_k t-F_k}( \xi_k)\I{T> k}\bigr] -r^{-\beta} \mathbf{E}
\bigl[\tilde P_{M_k r-F_k}( \xi_k)\I{T> k}\bigr] \bigr|
\nonumber
\\
&&\qquad\quad{} + 2^{k+1} r^{-\beta}.\nonumber
\end{eqnarray}
To complete the proof of Lemma~\ref{lemuniformmiddle}, we now devise
explicit bounds for the two main terms in \eqref{eqmainmiddle} when
we can ensure that coupling occured by level $k$ (i.e., $T\le k$) or not.

(i) \emph{No coupling by level $k$, $T>k$.} In this case, we
bound the terms roughly. We obtain
\begin{eqnarray*}
&&2^k \bigl\llvert t^{-\beta} \mathbf{E} \bigl[\tilde
P_{M_k t-F_k}( \xi_k)\I{T> k}\bigr] -r^{-\beta
}
\mathbf{E} \bigl[\tilde P_{M_k r-F_k}( \xi_k)\I{T> k}\bigr] \bigr
\rrvert
\\
&&\qquad\le2^{k+1}\sup_{u\ge t} u^{-\beta} \mathbf{E}
\bigl[\tilde P_{M_k u-F_k}(\xi_k)\I {T> k}\bigr].
\end{eqnarray*}
One then essentially uses the uniform bound $\sup_s\sup_u u^{-\beta}\Ec
{P_{u}(s)}\le C_5$ (see~(10) in \cite{CuJo2010}) and H\"older's and
Markov's inequalities to leverage a bound that makes profit of the
exponential tails of $T$. The details are found in \cite{CuJo2010},
page 16. For any $u>0$ and $s\in(\delta,1/2]$, one has
\begin{eqnarray*}
&& u^{-\beta} 2^k \mathbf{E} \bigl[\tilde P_{M_k u-F_k}(
\xi_k) \I{T>k}\bigr] \\
&&\qquad\le C_5 2^{k}
s^{-1/p} \biggl(\frac2{(\beta p +1) (\beta p +2)}\biggr)^{(k-1)/p}
\biggl( \frac{
\mathbf{E} [1.15^T]}{1.15^k}\biggr)^{1-1/p}
\\
&&\qquad\le C_4 C_5 \delta^{-1/2-1/(2p)}\biggl(2 \biggl\{
\frac2 {(\beta p +1) (\beta p + 2)} \biggr\}^{1/p} 1.15^{1/p-1}
\biggr)^k,
\end{eqnarray*}
by the upper bound in \eqref{eqboundK}.
Choosing $p$ close enough to one that the term in the brackets above is
strictly less than one, we obtain for any $s\in(\delta,1/2]$ and real
numbers $t,r>0$,
%
\begin{eqnarray}
\label{eqTgek}
&&2^k \bigl|t^{-\beta} \mathbf{E} \bigl[\tilde
P_{M_k t-F_k}( \xi_k)\I{T> k}\bigr] -r^{-\beta
}
\mathbf{E} \bigl[\tilde P_{M_k r-F_k}( \xi_k)\I{T> k}\bigr] \bigr| \nonumber\\
&&\qquad
\le2 C_4 C_5 \delta^{-1/2-1/(2p)}(1-
\gamma)^k
\\
&&\qquad\le C_1 \delta^{-1} (1-\gamma)^k,\nonumber
\end{eqnarray}
where $C_1$ denotes a constant and $\gamma>0$ (and $p>1$ is now fixed).

(ii) \emph{Coupling has occurred before level $k$, $T\le k$.}
In this case, we need to be a little more careful and match some terms.
In what follows, we write $x_+=x \vee0$. We start with
\[
t^{-\beta} 2^k \mathbf{E} \bigl[\tilde P_{M_k t-F_k}(
\xi_k)\I{T\le k}\bigr] =2^k \mathbf{E} \bigl[\I {T\le k}
\bigl(M_k -t^{-1} F_k \bigr)^\beta_+
\theta(M_k t-F_k)\bigr],
\]
where $\theta(x)=x^{-\beta}_+ \Ec{P_x(X)}$ with $X$ a $[0,1]$-uniform
random variable independent of everything else. The estimate in \eqref
{hwangquad} is easily transferred to the Poissonized version, and we
have $\theta(x)=\kappa+O(x^{-\eta})$
for any $ 0 < \eta< \beta.$
Therefore
%
\begin{eqnarray}
\label{eqKgekinter} &&2^k \bigl|t^{-\beta} \mathbf{E} \bigl[\tilde
P_{M_k t-F_k}(\xi_k)\I{T\le k}\bigr] -r^{-\beta
}
\mathbf{E} \bigl[\tilde P_{M_k r-F_k}(\xi_k)\I{T\le k}\bigr] \bigr|
\nonumber\\
&&\qquad\le2^k \bigl| \mathbf{E} \bigl[\I{T\le k} \bigl(M_k
-t^{-1} F_k \bigr)^\beta_+ \theta(M_k
t-F_k)\bigr]
\nonumber
\\[-8pt]
\\[-8pt]
\nonumber
&&\qquad\qquad{} - \mathbf{E} \bigl[\I{T\le k} \bigl(M_k -r^{-1}
F_k \bigr)^\beta_+ \theta(M_k
r-F_k)\bigr] \bigr|
\nonumber
\\
&&\qquad\le2^k \mathbf{E} \bigl[\bigl| \bigl(M_k -t^{-1}
F_k \bigr)^\beta_+ \theta(M_k
t-F_k) - \bigl(M_k -r^{-1} F_k
\bigr)^\beta_+ \theta(M_k r-F_k) \bigr| \bigr].
\nonumber
\end{eqnarray}
Fix $\eta< \beta$. For $x>0$, we have, as $x\to\infty$
\begin{eqnarray*}
&& \bigl(M_k -x^{-1} F_k
\bigr)^\beta_+ \cdot\theta(M_k x-F_k)
\\
&&\qquad=M_k^\beta \bigl(1-O \bigl(x^{-1}F_k
M_k^{-1} \bigr) \bigr) \bigl(\kappa+O
\bigl(M_k^{-\eta} x^{-\eta
} \bigr) \bigr)
\\
&&\qquad=\kappa M_k^{\beta} + O \bigl(F_k
M_k^{\beta-1} x^{-1} \bigr) + O \bigl(M_k^{\beta-\eta}
x^{-\eta} \bigr) + O \bigl(F_k M_k^{\beta-1-\eta}
x^{-1-\eta} \bigr)
\\
&&\qquad=\kappa M_k^{\beta} + O \bigl(F_k
M_k^{\beta-1} x^{-1} \bigr) + O \bigl(x^{-\eta}
\bigr) + O \bigl(F_k M_k^{\beta-1-\eta}
x^{-1-\eta} \bigr),
\end{eqnarray*}
since $M_k\in(0,1)$ and $\eta<\beta$, the $O(\cdot)$ terms being
deterministic and uniform in $s \in[0,1]$. Going back to \eqref{eqKgekinter},
the terms $\kappa M_k^\beta$
coming from the two terms with $t$ and $r$ cancel out, and there exist
constants $C_7, C_8$ such that, for all $t,r$ large enough
such that moreover $t\le r$, we have
\begin{eqnarray*}
&&2^k\bigl |t^{-\beta} \mathbf{E} \bigl[\tilde P_{M_k t-F_k}(
\xi_k)\I{T\le k}\bigr] -r^{-\beta
} \mathbf{E} \bigl[\tilde
P_{M_k r-F_k}( \xi_k)\I{T\le k}\bigr] \bigr|
\\
&&\qquad\le C_7 2^k \bigl(t^{-1} \mathbf{E}
\bigl[F_k M_k^{\beta-1}\bigr] +t^{-\eta} +
t^{-1-\eta
} \mathbf{E} \bigl[F_k M_k^{\beta-1 -\eta}
\bigr]\bigr)
\\
&&\qquad\le C_8 2^k t^{-\eta} \mathbf{E}
\bigl[F_k M_k^{\beta-1 -\eta}\bigr].
\end{eqnarray*}
Since it will be necessary to choose $k$ tending to infinity with $r$
to control the term in \eqref{eqTgek}, it remains to estimate $\Ec{F_k
M_k^{\beta-1-\eta}}$.
By definition of $F_k=M_k \tau_k$, one easily verifies that $F_k\le\sum_{i=1}^k \zeta_k$, where the normalized inter-arrival times $\zeta_i$
were defined right after
\eqref{eqkunfold0}.
Since $M_i\le1$ for every $i$, we have
\begin{eqnarray*}
\mathbf{E} \bigl[F_k M_k^{\beta-1 -\eta}\bigr] &\le& k
\mathbf{E} \bigl[M_k^{\beta-1-\eta}\bigr]
\\
&\le& k \mathbf{E} \bigl[X^{\beta-1-\eta}\bigr] ^{2k}= k (\beta-
\eta)^{-2k},
\end{eqnarray*}
by the lower bound on $M_k$ in Lemma~\ref{lemvollower}, $X$ denoting
a uniform on $[0,1]$. We finally obtain
%
\begin{eqnarray}
\label{eqTlek} &&2^k\bigl |t^{-\beta} \mathbf{E} \bigl[\tilde
P_{M_k t-F_k}( \xi_k)\I{T\le k}\bigr] -r^{-\beta
}
\mathbf{E} \bigl[\tilde P_{M_k r-F_k}( \xi_k)\I{T\le k}\bigr]\bigr |
\nonumber
\\[-8pt]
\\[-8pt]
\nonumber
&&\qquad\le C_8 k t^{-\eta} 2^k (\beta-
\eta)^{-2k}.
\end{eqnarray}
Putting \eqref{eqTgek} and \eqref{eqTlek} together with \eqref
{eqmainmiddle} yields, for any $t,r>0$ such that $t\le r$
\begin{eqnarray*}
&&\bigl| t^{-\beta} \mathbf{E} \bigl[P_t(s)\bigr] -r^{-\beta}
\mathbf{E} \bigl[P_r(s)\bigr] \bigr| \\
&&\qquad\le C_1
\delta^{-1} (1- \gamma)^k + C_8 k
2^k (\beta- \eta)^{-2k} t^{-\eta}+2^{k+1}
t^{-\beta}
\\
&&\qquad\le C_1 \delta^{-1} (1-\gamma)^k +
C_2 k 2^k (\beta-\eta)^{-2k} t^{-\eta}
\end{eqnarray*}
for some constant $C_2$. The statement in Lemma~\ref
{lemuniformmiddle} follows readily from the triangle inequality.
\end{appendix}


%


\printaddresses

\end{document}